\documentclass[11pt]{amsart}

\usepackage{fullpage}

\usepackage[T1]{fontenc}
\usepackage{lmodern}
\usepackage{microtype}
\usepackage[a4paper,margin=1in]{geometry}
\usepackage{amsmath,amssymb,amsthm,mathtools}
\usepackage{enumitem}
\usepackage{xcolor}
\definecolor{hot}{RGB}{65,105,225}
\usepackage[pagebackref=true,colorlinks=true, linkcolor=hot ,  citecolor=hot, urlcolor=hot]{hyperref}
\usepackage{microtype}

\newtheorem{theorem}{Theorem}[section]
\newtheorem{proposition}[theorem]{Proposition}
\newtheorem{lemma}[theorem]{Lemma}
\newtheorem{corollary}[theorem]{Corollary}
\theoremstyle{definition}
\newtheorem{remark}[theorem]{Remark}

\newtheorem{example}[theorem]{Example}
\DeclareMathOperator{\Sym}{Sym}
\DeclareMathOperator{\rank}{rank}

\DeclareMathOperator{\codim}{codim}

\DeclareMathOperator{\Bl}{Bl}
\DeclareMathOperator{\Gr}{Gr}

\DeclareMathOperator{\Tot}{Tot}
\DeclareMathOperator{\PGL}{PGL}

\DeclareMathOperator{\Mat}{Mat}
\DeclareMathOperator{\diag}{diag}
\DeclareMathOperator{\Hom}{Hom}

\DeclareMathOperator{\GL}{GL}
\DeclareMathOperator{\id}{id}
\DeclareMathOperator{\coker}{coker}

\newcommand{\cO}{\mathcal O}
\newcommand{\PP}{\mathbf P}
\newcommand{\CC}{\mathbf C}
\newcommand{\OO}{\mathcal O}
\newcommand{\cF}{\mathcal F}
\newcommand{\cK}{\mathcal K}
\newcommand{\cG}{\mathcal G}

\newcommand{\cZ}{\mathcal Z}
\newcommand{\A}{\mathbb{A}}
\newcommand{\C}{\mathbb{C}}

\newcommand{\wedgeTwo}{\bigwedge\nolimits^2}

\newcommand{\Ztop}{Z^{\mathrm{top}}}

\title{Local topological zeta functions of generic symmetric and catalecticant determinantal varieties}
\author{}
\date{}

\author{An-Khuong Doan}\address{Institute for Advanced Studies in Mathematics, Harbin Institute of Technology, Harbin, China 150001}\email{an-khuong.doan@hit.edu.cn}
\begin{document}
\begin{abstract}
We give explicit formulas for the local topological zeta functions at the origin of generic symmetric determinantal varieties and catalecticant determinantal varieties associated with generic Hankel matrices. In the symmetric determinant hypersurface case, the formula simplifies to a closed product expression and implies the Strong Monodromy Conjecture. For catalecticant determinantal varieties, we establish a simultaneous local reduction of the determinantal ideals and obtain a closed product formula for their local topological zeta functions. Together with the known Bernstein–Sato polynomial of the Hankel determinant, this also implies the Strong Monodromy Conjecture in the square catalecticant determinant case. The method used in this paper also gives short proofs of Chen–Zuo/Docampo's formula in the generic determinantal case and Chen–Shi–Xu–Zuo's formula in the skew-symmetric case.

\end{abstract}
\setcounter{tocdepth}{1}
\maketitle
\tableofcontents
\section{Introduction}
Topological zeta functions, introduced by Denef and Loeser \cite{DL92}, provide a resolution-theoretic invariant of singularities which lies at the intersection of birational geometry, singularity theory, and the theory of Bernstein--Sato polynomials.  Explicit computations are known only for rather special classes of
singularities. For hypersurfaces which are nondegenerate with respect to their Newton polyhedron, the topological zeta function admits a Newton-polyhedral description; see, for instance, \cite{DL92} and the references therein. Related Newton-polyhedral formulas and descriptions of candidate poles for ideals
and analytic mappings were developed by Veys and Z\'u\~niga-Galindo \cite{VZ08}. Hyperplane arrangements form another important class in which the resolution geometry is sufficiently combinatorial to permit explicit calculations and to study the monodromy conjecture; see, for example, Budur--Musta\c{t}\u{a}--Teitler
\cite{BMT11}.

A particularly relevant result is provided by generic determinantal varieties. Let $M=\mathbb A^{mn}$ be the space of $m\times n$ matrices and let $Z_r\subset M$ be defined by the $r\times r$ minors of the generic matrix. Docampo studied the arc spaces and jet schemes of these varieties using the natural $\mathrm{GL_m}\times\mathrm{GL}_n$-action and, through motivic integration, obtained explicit formulas for their topological zeta functions \cite[Theorem 6.5]{Doc13}. In the square case $m=n$, one has $$ Z^{\mathrm{top}}_{Z_r}(s) = \prod_{j=1}^{r} \frac{1}{ 1+s\,\dfrac{r+1-j}{(m+1-j)^2}}. $$
Recently, Chen-Zuo have pointed out that the same argument can be extended to the case $m\neq n$ in \cite[Theorem 5.4]{CH25}, giving
$$ Z^{\mathrm{top}}_{Z_r}(s) = \prod_{j=1}^{r} \frac{1}{ 1+s\,\dfrac{r+1-j}{(m+1-j)(n+1-j)}}.
$$  Finally,  Chen-Shi-Xu-Zuo provide a closed formula in the skew-symmetric generic determinantal case (cf. \cite[\S3.8]{CSXZ25}).

The aim of the present paper is to obtain an analogous formula for the two remaining  important classes: the class of \textit{generic symmetric determinantal varieties} and that of \textit{catalecticant determinantal varieties}, which turn out to be more involved.

Although the rank stratification has a similar recursive character, the geometry of the resolutions is different. A key ingredient in the symmetric case is a simultaneous local reduction of the determinantal ideals, based on the Schur complement, which allows us to reduce the problem locally to that of a smaller generic symmetric matrix (cf. Lemma \ref{lem:simultaneous-local-reduction}). Successive blow-ups of the symmetric rank loci lead naturally to the geometry of complete quadrics, and the strata over the origin are governed by flags of kernels together with nondegenerate quadratic forms on their successive quotients. A crucial feature of the symmetric case is the vanishing of the Euler characteristic of the space of nondegenerate projective quadratic forms in dimension at least three (cf. Lemma \ref{lem:Qchi}). This produces a strong cancellation among the resolution strata and ultimately reduces the computation to a combinatorics of paths with steps of length one and two. For catalecticant determinantal varieties, we consider generic Hankel matrices, whose projectivized rank loci are precisely the secant varieties of a rational normal curve. The successive blow-ups of these rank loci are closely related to the classical resolution of secant varieties constructed by Bertram (cf. \cite{Ber92}). In this setting, a reciprocal-series transformation provides a simultaneous local reduction of all determinantal ideals to those of a smaller generic Hankel matrix (cf. Lemma \ref{lem:cat-local-reduction}). This allows us to establish the corresponding log-principalization and compute its numerical data explicitly. Moreover, the same local reduction gives a recursive description of the boundary strata over the origin. The Euler-characteristic computation is particularly simple, i.e. we show that every relevant projective rank stratum has Euler characteristic two (see Lemma \ref{lem:cat-rank-euler}). Unlike in the symmetric case, there is no analogous vanishing of the Euler characteristics of the boundary strata, and every stratum indexed by a subset containing the first boundary divisor contributes to the topological zeta function. As a consequence, the resulting sum admits a direct factorization, yielding a closed product formula for the local topological zeta function. Let us now state the main results.

We deal with the symmetric case first. Let $V$ be a complex vector space of dimension $n$ and let $S_0:=\Sym^2 V^\vee\simeq \mathbf A^N$, and  $N=\binom{n+1}{2}.$
After choosing a basis of $V$, let $ X_n=(x_{ij})_{1\le i,j\le n},$ $ x_{ij}=x_{ji},$ be the generic symmetric $n\times n$ matrix. For $1\le k\le n$, set $$
S_k=\{A\in S_0:\rank A\le n-k\}.$$
Put $ r:=n-k+1.$  Then $S_k$ is defined by the ideal $I_k:=I_r(X)$ generated by the $r\times r$ minors of $X$. In what follows, it is convenient to use the rank indexing
$$
Y_j:=\{A\in S_0:\rank A\le j-1\}=V(I_j(X)),
\qquad 1\le j\le n.
$$
Thus, $Y_1=\{0\}$ and $S_k=Y_r.$ For $1\le j\le r$, define
$$ D_j(s):=(r-j+1)s+\binom{n-j+2}{2}$$
Define recursively
\begin{equation}\label{eq:Rinit}
R_0(s)=0,\qquad R_1(s)=1,
\end{equation}
and, for $j\ge 2$,
\begin{equation}\label{eq:Rrec}
R_j(s)=R_{j-1}(s)+\frac{D_{j-1}(s)}{2}R_{j-2}(s).
\end{equation}

\begin{theorem}\label{thm:main}
With the notation above, the local topological zeta function of the ideal $I_k$ at the origin is
\begin{equation}\label{eq:main}
Z^{\mathrm{top}}_{(S_0,S_k),0}(s) = \binom{k+1}{2}\frac{n!}{k!} \frac{R_r(s)+\dfrac{D_r(s)}{k+1}R_{r-1}(s)}{\displaystyle\prod_{j=1}^{r}D_j(s).}
\end{equation}
When $r=1$, the term involving $R_{r-1}=R_0$ is zero.

Equivalently, let $$\mathcal P_b = \left\{ (1=j_0<j_1<\cdots<j_t=b): j_\ell-j_{\ell-1}\in\{1,2\} \right\},$$ and for $\mathbf j=(j_0,\dots,j_t)\in\mathcal P_b$ let
$$q(\mathbf j) = \#\{\ell:j_\ell-j_{\ell-1}=2\}.$$
Then
\begin{align}\label{eq:pathformula}
Z^{\mathrm{top}}_{(S_0,S_k),0}(s)={}& \binom{k+1}{2}\frac{n!}{k!} \sum_{\mathbf j\in\mathcal P_r} \frac{2^{-q(\mathbf j)}}{\prod_{\ell=0}^{t}D_{j_\ell}(s)} \nonumber\\
&+ \binom{k+1}{2}\frac{n!}{(k+1)!} \sum_{\mathbf j\in\mathcal P_{r-1}} \frac{2^{-q(\mathbf j)}}{\prod_{\ell=0}^{t}D_{j_\ell}(s)},
\end{align}
where the second sum is omitted when $r=1$.
\end{theorem}
\begin{corollary}\label{cor:main} The local topological zeta function of the pair $(S_0,S_1)$ at the origin is given as follows:
\begin{itemize}
\item[1.] If $n=2m+1$ is odd, 
$$ Z^{\mathrm{top}}_{(S_0,S_1),0}(s) = \frac{1\cdot 3\cdot 5\cdots n} {(s+1)\displaystyle\prod_{j=1}^{m}(2s+2j+1)}.
$$

\item[2] If $n=2m$ is even, 
$$
Z^{\mathrm{top}}_{(S_0,S_1),0}(s) = \frac{(n-1)!!\,(s+n+1)} {(s+1)\displaystyle\prod_{j=1}^{n/2}(2s+2j+1)}.
$$
\end{itemize}
Consequently, the Strong Monodromy Conjecture holds for the pair $(S_0,S_1)$.
\end{corollary}

Next, we turn to the catalecticant case. Let $m,n\geq1$, put $d=m+n-2$, and set $S_0:=\A^{d+1}$, with coordinates $x_0,\ldots,x_d$.  Consider the generic catalecticant matrix
$$ H_{m,n}(x):=(x_{i+j})_{\substack{0\leq i\leq m-1\\0\leq j\leq n-1}} =
\begin{pmatrix}
x_0&x_1&\cdots&x_{n-1}\\
x_1&x_2&\cdots&x_n\\
\vdots&\vdots&&\vdots\\
x_{m-1}&x_m&\cdots&x_{m+n-2}
\end{pmatrix}.
$$
For $1\leq j\leq\min\{m,n\}$, let $I_j:=I_j(H_{m,n})$ be the ideal generated by its $j\times j$ minors, and write
$$
Y_j:=V(I_j)=\{x\in S_0:\rank H_{m,n}(x)\leq j-1\}.
$$
Recall that the ideals $I_j$ are prime and that
\begin{equation}
\label{eq:cat-dimension}
\dim Y_j=2j-2,
\qquad
\codim_{S_0}Y_j=d-2j+3=m+n-2j+1;
\end{equation}
by \cite[\S1]{CMSV18}.  For $j\geq2$, $\PP(Y_j)$ is the closure of the union of the $(j-2)$-planes spanned by $j-1$ points of the rational normal
curve in $\PP^d$.  
Denote $$ c:=m+n-2r+1 \text{ and } D_j(s):=(r-j+1)s+m+n-2j+1 \;(1\leq j\leq r).$$
Notice that $c\geq1$ and that $c=1$ if and only if $m=n=r$.

\begin{theorem}
\label{thm:cat-zeta}
With the notation above, the local topological zeta function of $I_r$ at
the origin is
\begin{equation}
\label{eq:cat-zeta-product}
\Ztop_{I_r,0}(s)
=
\frac{c}{D_1(s)}
\prod_{j=2}^{r}\left(1+\frac{2}{D_j(s)}\right).
\end{equation}
Equivalently,
\begin{equation}
\label{eq:cat-zeta-reindexed}
\Ztop_{I_r,0}(s)
=
c\,
\frac{\displaystyle\prod_{q=1}^{r-1}\bigl(q(s+2)+c\bigr)}
     {\displaystyle\prod_{q=1}^{r}\bigl(q(s+2)+c-2\bigr)}.
\end{equation}
Here, $c:=m+n-2r+1$ and empty products are understood to be $1$.
\end{theorem}
\begin{corollary}
\label{cor:cat-determinant}
Let $f_n:=\det H_{n,n}$ on $\A^{2n-1}$.  Then
\begin{equation}
\label{eq:cat-determinant-zeta}
\Ztop_{f_n,0}(s)
=
\frac{\displaystyle\prod_{q=1}^{n-1}\bigl(q(s+2)+1\bigr)}
     {\displaystyle\prod_{q=1}^{n}\bigl(q(s+2)-1\bigr)}.
\end{equation}
Its poles are all simple and are exactly $$-2+\frac1q,\qquad q=1,\ldots,n.$$ Consequently, the Strong Monodromy Conjecture holds
for the pair $(H_{n,n},Y_n)$ at the origin.
\end{corollary}
Finally, exactly the same method used in this paper can be applied to the case of generic determinantal varieties and that of skew-symmetric one thereby obtaining a short proof of Chen-Zuo/Docampo's topological zeta function formula and Chen-Shi-Xu-Zuo's one. We briefly record this fact in the last two sections.

During the final preparation of this manuscript, the author became aware of recent independent work by Lyu and Zhang \cite{LZ26}, who computed the local and global topological zeta functions of the generic symmetric determinant hypersurface and established the Strong Monodromy Conjecture in this case. Their formula for the local topological zeta function agrees with the determinant-hypersurface specialization of our general symmetric determinantal formula. The present paper treats arbitrary generic symmetric determinantal ideals and also includes generic catalecticant determinantal ideals associated with Hankel matrices.
\subsection*{Acknowledgement} The author thanks Nero Budur and Huu-Kien Nguyen for discussions. This work was supported by HIT Grant AUGA5610320926. 
\subsection*{AI disclosure} During the preparation of this manuscript, the author used ChatGPT Plus to assist with mathematical exploration, verification of computations, literature review and revisions. The author takes full responsibility for the correctness, originality, and integrity of the results presented in this paper.

\section{Topological zeta functions and embedded resolutions}

We recall the definition of the local topological zeta function associated with an ideal, together with the resolution-theoretic data entering its construction. Let $X$ be a smooth complex algebraic variety of dimension $N$, let $
\mathfrak a\subset \mathcal O_X $
be a nonzero coherent ideal sheaf, and let $x\in X$ be a closed point.
\subsection{Log-principalizations}
A \emph{log-principalization} of $\mathfrak a$ is a proper birational morphism $\pi:Y\longrightarrow X $
such that:
\begin{enumerate}
\item $Y$ is smooth;
\item $\pi$ is an isomorphism over $X\setminus V(\mathfrak{a})$;
\item the pull-back ideal is locally principal, $\mathfrak a\mathcal O_Y = \mathcal O_Y(-F), $ for an effective divisor $F=\sum_{i\in T}N_iE_i,$, $N_i\in\mathbb N_{>0};
$ \item the divisor $ \operatorname{Supp}(F)$ has simple normal crossings.
\end{enumerate}

Here the $E_i$, $i\in T$, denote the irreducible components of the divisor occurring in the principalization. The relative canonical divisor can be written as
$$ K_{Y/X} = K_Y-\pi^*K_X = \sum_{i\in T}(\nu_i-1)E_i. $$ The integers $(N_i,\nu_i) $ are called the \emph{numerical data} of the principalization.
Equivalently, $ N_i=\operatorname{ord}_{E_i}(\mathfrak a), $ while $ \nu_i-1 =\operatorname{ord}_{E_i}(K_{Y/X}). $

For a nonempty subset $
I\subseteq T,
$ set  $ E_I := \bigcap_{i\in I}E_i \text { and } E_I^\circ := E_I \setminus \bigcup_{j\notin I}E_j.$
Thus $E_I^\circ$ is the stratum consisting of points which lie precisely on the divisors $E_i$ with $i\in I$.
\subsection{The local topological zeta function}
The \emph{local topological zeta function} of the ideal $\mathfrak a$ at $x$ is defined by
$$ Z^{\mathrm{top}}_{\mathfrak a,x}(s) = \sum_{\varnothing\neq I\subseteq T} \chi\!\left( E_I^\circ\cap\pi^{-1}(x) \right) \prod_{i\in I} \frac{1}{N_is+\nu_i}
$$
(cf. \cite{DL92}). Here $\chi(-)$ denotes the topological Euler characteristic with compact support. For complex algebraic varieties this agrees with the ordinary Euler characteristic in all situations occurring below.

A fundamental fact is that $Z^{\mathrm{top}}_{\mathfrak a,x}(s)$ is independent of the chosen log-principalization. Hence it is an invariant of the germ of the pair $(X,V(\mathfrak a))$ at $x$.

The numbers $-\frac{\nu_i}{N_i}$ are called the \emph{candidate poles} supplied by the resolution. Some of them may cancel in the final expression, so the actual poles of $Z^{\mathrm{top}}_{\mathfrak a,x}(s)$ form, in general, a proper subset of the candidate poles.

\subsection{Embedded resolutions}

Suppose now that $Z\subset X$ is a reduced closed subvariety with ideal sheaf $\mathcal I_Z.$ An \emph{embedded resolution} of $Z$ in $X$ is a proper birational morphism $\pi:Y\longrightarrow X$
such that:
\begin{enumerate}
\item $Y$ is smooth;
\item $\pi$ is an isomorphism over $X\setminus sing(Z)$;
\item the strict transform $\widetilde Z\subset Y$ is smooth;
\item the union $\widetilde Z\cup \operatorname{Exc}(\pi)$ has normal crossings.
\end{enumerate}

One must distinguish an embedded resolution from a log-principalization of the ideal $\mathcal I_Z.$ An embedded resolution resolves the geometry of the reduced subvariety $Z$, but it need not make the ideal $\mathcal I_Z\mathcal O_Y$ invertible. If the strict transform $\widetilde Z$ has codimension strictly greater than one, then the pulled-back ideal is generally still nonprincipal along $\widetilde Z$. In good situations for example when $\mathcal{I}_Z\mathcal O_Y=\mathcal O_Y(-D)\cdot\mathcal{I}_{\tilde Z}$ where $D$ is an simple normal crossing divisor and the $\tilde Z$ meet $D$ with normal crossings, one can obtain a log-principalization by blowing up the smooth strict transform $\widetilde Z$ once more: $\widehat Y = \operatorname{Bl}_{\widetilde Z}(Y) \longrightarrow Y.$
Since $\widetilde Z$ meets the exceptional divisor with normal crossings, this additional blow-up preserves the normal-crossing structure and principalizes the ideal. If, on the other hand, $Z$ is a hypersurface, then the strict transform $\widetilde Z$ is already a divisor. In that case an embedded resolution is already sufficient for the principal ideal $(f)$, and one simply includes the strict transform among the divisors entering the topological zeta function.
\begin{example}
\label{ex:embedded-not-principalization}
In general, an embedded resolution of a reduced subvariety need not become a
log-principalization of its ideal after blowing up the smooth strict transform. Indeed, let $X=\A^3_{\C},$ and $I=(q,c),$
where $q=xy-z^2,$ and $c=x^3+y^3+z^3,$ and set $Z=V(I).$ Hence $Z$ is the reduced union of six distinct lines through the origin.

Let $\pi:\tilde X=\Bl_0(X)\longrightarrow X$ be the blow-up of $X$ at the origin. The strict transforms of these six lines are disjoint and smooth, and each meets the exceptional divisor $E\simeq\PP^2$ transversely. Thus $\pi$ is an embedded resolution of $Z$. 

As usual, $\widetilde X \subset \A^3\times\PP^2$ is the closed subvariety with coordinates
$(x,y,z)\times[t:u:v]$ defined by $xu=yt,\; xv=zt,\; yv=zu.$ The exceptional divisor is $E=\pi^{-1}(0)=\{(0,0,0)\}\times\PP^2.$ Consider the affine chart $U_t:=\{t\neq0\}.$ Putting $t=1$, the equations of the blow-up give
$y=xu,\; z=xv.$ Thus $U_t\simeq\A^3$ with coordinates $(x,u,v)$, and $\pi(x,u,v)=(x,xu,xv).$
Moreover, $E\cap U_t=(x=0).$

On this chart, $\pi^*q = x(xu)-(xv)^2= x^2(u-v^2),$ and $\pi^*c=x^3+(xu)^3+(xv)^3=x^3(1+u^3+v^3).$ Set $q':=u-v^2,$ and $
c':=1+u^3+v^3.$ Then
$$
I\OO_{\widetilde X}|_{U_t}=(x^2q',x^3c')=x^2(q',xc').
$$
and  the strict transform $\widetilde Z$ is therefore locally defined by $I_{\widetilde Z}=(q',c').$

Choose a point $p\in E$ such that $q'(p)=0,$ and $c'(p)\neq0.$ Then $c'$ is a unit in $\OO_{Y,p}$, and hence locally near $p$, $I\OO_{Y,p}=x^2(q',x).$ Since $x$ and $q'$ are independent local parameters at $p$, the ideal $(q',x)$ has codimension $2$ and is not principal. Thus $I\OO_Y$ is not invertible near $p$. Moreover, by choice, $p\notin\widetilde Z.$  
Hence, blowing up the smooth strict transform $\widetilde Z$ is an isomorphism near $p$ and cannot remove this
nonprincipal locus. 

For completeness, the failure can also be seen above points lying on the strict transform itself. Let $p_0\in \widetilde Z\cap E.$ Then $x(p_0)=q'(p_0)=c'(p_0)=0.$ Since the conic $q'=0$ and the cubic $c'=0$ meet transversely at the
six points of $\widetilde Z\cap E$, and $\widetilde Z$ meets $E=(x=0)$ transversely, the functions $x,\; q',\; c'$
form a regular system of local parameters at $p_0$. Consider now the blow-up $
\beta:
\widehat X
=
\Bl_{\widetilde Z}(\widetilde X)
\longrightarrow
\widetilde X.
$
Locally at $p_0$, this is the blow-up of the ideal $ I_{\widetilde Z}=(q',c').$ On the $q'$-chart, write $
c'=q'w.$ Then $(q',xc')=(q',xq'w)=(q'),$ so the pull-back of $I\OO_{\widetilde X}$ is principal on this chart.
On the $c'$-chart, however, write $q'=c'w.$ Then $(q',xc')=(c'w,xc')=c'(w,x),$ and hence $I\OO_{\widehat X}=x^2c'(w,x).$ At the point above $p_0$ defined by $x=c'=w=0,$ the functions $x,c',w$ form regular local parameters. Therefore the residual ideal $(w,x)$ has codimension $2$ and is not principal. Thus the blow-up along $\widetilde Z$ still does not principalize $I$, even above certain points of the center itself.

\end{example}

\subsection{Blowing up a smooth center} We shall repeatedly use the numerical data associated with the blow-up of a smooth center. Let $ \pi:\widetilde X=\operatorname{Bl}_C(X)\longrightarrow X
$ be the blow-up of a smooth subvariety $C\subset X$ of codimension $c$, and let $E$ be the exceptional divisor. Then $K_{\widetilde X/X}=(c-1)E.$ Hence the corresponding log-discrepancy coefficient is $\nu(E)=c.$

If $\mathfrak a$ is an ideal on $X$ and $m=\operatorname{ord}_C(\mathfrak a),$ then $\mathfrak a\mathcal O_{\widetilde X}=\mathcal O_{\widetilde X}(-mE)\cdot \widetilde{\mathfrak a},$ where $\widetilde{\mathfrak a}$ denotes the weak transform. Therefore, $N(E)=\operatorname{ord}_C(\mathfrak a).$
Thus, whenever an exceptional divisor $E$ is created by blowing up a smooth center $C$ which is not contained in previously created exceptional divisors, its numerical data are simply
$$
N(E)=\operatorname{ord}_C(\mathfrak a),
\qquad
\nu(E)=\operatorname{codim}_X C.
$$
In the applications below, the centers form a nested chain $$Y_1\subset Y_2\subset\cdots\subset Y_r\subset X.$$
We blow up successively the strict transforms of $Y_1,Y_2,\dots,Y_{r-1},$ and, if necessary, finally the smooth strict transform of $Y_r$. The crucial point is that, at the moment when the strict transform of $Y_j$ is blown up, its generic point lies over $Y_j\setminus Y_{j-1}.$ All earlier blow-ups are isomorphisms over this open stratum. Consequently, the generic point of the center does not lie on the previously created exceptional divisors. Therefore its numerical data are computed directly from the original geometry: $N_j=\operatorname{ord}_{Y_j}(\mathfrak a),$ and $\nu_j=\operatorname{codim}_X(Y_j).$
This observation is particularly useful for generic determinantal and symmetric determinantal ideals, because both quantities can be computed explicitly from the rank stratification.

\section{The Gaube--Schober resolution: principalization and numerical data}
Let $V$ be a complex vector space of dimension $n$, and put $S_0:=\Sym^2 V^\vee.$ After choosing a basis of $V$, let $
X_n=(x_{ij})_{1\leq i,j\leq n},\; x_{ij}=x_{ji}, $ be the generic symmetric matrix, and denote by $I_p^{(n)}:=I_p(X_n)
$ the ideal generated by its $p\times p$ minors. We write $$Y_p^{(n)}:=V(I_p^{(n)}) = \{A\in S_0:\rank A\leq p-1\}.
$$
\begin{remark} As the reader might already have noticed, in these notations, we use one more superscript $n$ compared to those in the introduction. The main aim is to emphasize the dimension $n$ of the ambient vector space $V$ which will be convenient for inductions $n$ in the sequel. However, the superscript $n$ is only used in this section.
\end{remark}
\begin{lemma}
\label{lem:simultaneous-local-reduction}
Let $A\in S_0$ have rank $d$.
\begin{enumerate}
\item[(i)]
Assume $d\geq 1$. After a change of basis of $V$, there exists a Zariski neighbourhood $U$ of $A$ and an isomorphism $U\simeq U_0\times \Sym^2\C^{\,n-d,\vee},$ where $U_0$ is smooth, with the following property. If
$B=(b_{\alpha\beta})$ denotes the generic symmetric $(n-d)\times(n-d)$ matrix on the second factor, then, 
$$
I_p^{(n)}\OO_U=
\begin{cases}
\OO_U,&p\leq d,\\[2mm]
I_{p-d}(B)\OO_U,&p>d.
\end{cases}
$$
Consequently,
$$
Y_p^{(n)}\cap U=
\begin{cases}
\varnothing,&p\leq d,\\[2mm]
U_0\times Y_{p-d}^{(n-d)},&p>d.
\end{cases}
$$

\item[(ii)] Let $\sigma:\widetilde S_0:=\Bl_0(S_n)\longrightarrow S_0$ be the blow-up of the origin, and let $E$ be its exceptional divisor. At every point of $E$, after a change of basis of $V$, there exists a Zariski neighbourhood $U\subset \widetilde S_0$ withcoordinates $t, u_2,\ldots,u_n, b_{ij}$, $(2\leq i\leq j\leq n)$
such that $E\cap U=(t=0)$ and, for every $p\geq 1$,
\begin{equation}
\label{eq:local-transform-minors}
I_p^{(n)}\OO_U
=
t^p I_{p-1}(B),
\end{equation}
where $B=(b_{ij})_{2\leq i,j\leq n}$ is a generic symmetric $(n-1)\times(n-1)$ matrix.
In particular, for $p\geq 2$, the strict transform of
$Y_p^{(n)}$ is defined on $U$ by
\begin{equation}
\label{eq:strict-transform-rank-locus}
I_{p-1}(B),
\end{equation}
and hence
\begin{equation}
\label{eq:local-product-rank-locus}
\widetilde Y_p^{(n)}\cap U
\simeq
\A^1\times\A^{n-1}\times Y_{p-1}^{(n-1)}.
\end{equation}
Moreover, these identifications hold simultaneously for all $p$. Thus the strict transforms of the whole chain $Y_2^{(n)}
\subset Y_3^{(n)}\subset\cdots\subset Y_r^{(n)} $ identify locally with the corresponding chain $Y_1^{(n-1)}\subset Y_2^{(n-1)}\subset\cdots\subset Y_{r-1}^{(n-1)}$ for the residual symmetric matrix $B$.
\end{enumerate}
\end{lemma}

\begin{proof}
We first prove~(i). Since $A$ is symmetric of rank $d$, after a congruence transformation we may assume that its upper-left
$d\times d$ principal block is invertible. Shrinking around $A$, write the universal symmetric matrix in block form
$$ M= \begin{pmatrix}
P&Q\\
Q^t&R
\end{pmatrix},
$$
where $P$ is a symmetric invertible $d\times d$ matrix. Set $B:=R-Q^tP^{-1}Q.$ The change of variables $(P,Q,R)\longmapsto(P,Q,B)$ is an isomorphism on the open set $\det P\neq 0$, with inverse $R=B+Q^tP^{-1}Q.$
Moreover, a direct computation shows
$$
\begin{pmatrix}
I_d&0\\
-Q^tP^{-1}&I_{n-d}
\end{pmatrix}
M
\begin{pmatrix}
I_d&-P^{-1}Q\\
0&I_{n-d}
\end{pmatrix}
=
\begin{pmatrix}
P&0\\
0&B
\end{pmatrix}.
$$
where the matrices on the left and on the right are both invertible. Consequently, we obtain
$$I_p(M) = I_p
\begin{pmatrix}
P&0\\
0&B
\end{pmatrix}.
$$
Hence, (i) follows.

We now prove~(ii). Consider first the $x_{11}$-chart of the blow-up of the origin. Set $t=x_{11},$ $x_{1i}=tu_i \;(i\geq 2),$ $x_{ij}=tc_{ij} \;(2\leq i\leq j\leq n).$ Then
$$\sigma^*X_n=t
\begin{pmatrix}
1&u^t\\
u&C
\end{pmatrix},
$$
where $u=(u_2,\ldots,u_n)^t,$ and $C=(c_{ij})_{2\leq i,j\leq n}.$ If we denote $B:=C-uu^t.$
Since $C=B+uu^t,$ the entries of $B$, together with $t,u_2,\ldots,u_n$, are again algebraically independent coordinates on this chart. Now, a direct calculation shows that 
$$
\begin{pmatrix}
1&0\\
-u&I_{n-1}
\end{pmatrix}
\begin{pmatrix}
1&u^t\\
u&C
\end{pmatrix}
\begin{pmatrix}
1&-u^t\\
0&I_{n-1}
\end{pmatrix}
=
\begin{pmatrix}
1&0\\
0&B
\end{pmatrix}.
$$
Hence, by part~(i) with $d=1$,
$$
I_p
\begin{pmatrix}
1&u^t\\
u&C
\end{pmatrix}
=I_{p-1}(B).
$$
so that $I_p^{(n)}\OO_U=t^pI_{p-1}(B),$ 
which proves~\eqref{eq:local-transform-minors}. For $p\geq 2$,  since $I_{p-1}(B)$ involves only the variables $b_{ij}$, \eqref{eq:strict-transform-rank-locus} and then \eqref{eq:local-product-rank-locus} also follow.

It remains only to see that charts of this form are sufficient locally at every point of the exceptional divisor. A point of $E$ is represented by a nonzero projective symmetric form $[q]\in\mathbb{P}(\Sym^2V^\vee).$ There exists $v\in V$ such that
$q(v,v)\neq 0$. After a change of basis sending the first basis vector to $v$, the point lies in an $x_{11}$-chart. The congruence action of $\GL(V)$ preserves every symmetric ranklocus $Y_p^{(n)}$ and therefore lifts to all the blow-ups considered below. Hence the preceding calculation applies locally at every point of $E$.

Finally, since
\eqref{eq:local-transform-minors}--\eqref{eq:local-product-rank-locus}
hold for every $p$ simultaneously, they identify the whole chain of
later rank strata with the corresponding chain for the residual
matrix $B$.
\end{proof}
\begin{remark}
\label{rem:simultaneous-local-reduction-other-types}
The simultaneous local-reduction principle of
Lemma~\ref{lem:simultaneous-local-reduction} remains valid, with the
evident modifications, for the generic determinantal and generic
skew-symmetric determinantal varieties considered in the last sections. For later use, we record this fact in the folowing tables.

\noindent
\textit{Local reduction at a point of positive rank.}
$$
\begin{array}{c|c|c}
\text{matrix type} & \text{rank of the point} & \text{residual problem} \\ \hline
\text{symmetric }n\times n & d & (n,p)\longmapsto(n-d,p-d) \\[1mm]
\text{generic }m\times n & d & (m,n,p)\longmapsto(m-d,n-d,p-d) \\[1mm]
\text{skew-symmetric }n\times n & 2d & (n,p)\longmapsto(n-2d,p-d).
\end{array}
$$
\noindent
\textit{Reduction on a chart of the blow-up of the zero matrix.}
$$
\begin{array}{c|c}
\text{matrix type} & \text{residual problem after blowing up the origin} \\ \hline
\text{symmetric} & (n,p)\longmapsto(n-1,p-1) \\[1mm] \text{generic }m\times n & (m,n,p)\longmapsto(m-1,n-1,p-1) \\[1mm]
\text{skew-symmetric} & (n,p)\longmapsto(n-2,p-1).
\end{array}
$$
\end{remark}
Gaube--Schober in \cite[Theorem B]{GS24} prove that the generic symmetric determinantal variety $Y_r$ admits an embedded resolution obtained by successively blowing up the strict transforms of $Y_1,\ Y_2,\ \ldots,\ Y_{r-1}.$
More precisely, let $X_0:=S_0$ and, for $1\le j\le r-1$, let $\pi_j:X_j\longrightarrow X_{j-1}$ be the blow-up of the strict transform $Y_j^{(j-1)}$ of $Y_j$ in $X_{j-1}$. Then $Y_r^{(r-1)}$ is smooth and meets the exceptional divisor with simple normal crossings. Let $ \pi^{(r-1)} := \pi_1\circ\cdots\circ\pi_{r-1} : X_{r-1}\longrightarrow S_n$ be the resulting morphism, and let $E_j\subset X_{r-1}$ denote the final strict transform of the exceptional divisor created by
$\pi_j$.

\begin{proposition}
\label{prop:factorizing-GS}
The total transform of $I_r^{(n)}$ on $X_{r-1}$ admits the
factorization
\begin{equation}
\label{eq:factorizing-GS}
I_r^{(n)}\OO_{X_{r-1}} = \OO_{X_{r-1}} \left( -\sum_{j=1}^{r-1}(r-j+1)E_j \right) \cdot I_{Y_r^{(r-1)}}.
\end{equation}
Here $I_{Y_r^{(r-1)}}$ denotes the ideal sheaf of the strict transform $Y_r^{(r-1)}\subset X_{r-1}$.
\end{proposition}

\begin{proof}
We argue by induction on $r$. For $r=1$, there is no blow-up before the strict transform and $I_1^{(n)}=I_{Y_1^{(n)}},
$ so there is nothing to prove.

Assume $r\geq 2$, and suppose that the assertion has been proved for smaller values of the determinantal index. Since the desired identity of coherent ideals is local on $X_{r-1}$, fix a point $\xi\in X_{r-1}$
and let $ A:=\pi^{(r-1)}(\xi)\in S_n.$ There is nothing to prove if $A\notin Y_r^{(n)}$, because then
$I_r^{(n)}$ is locally the unit ideal. We may therefore assume that $\rank A=d\leq r-1$ and then distinguish two cases.

\medskip

\noindent
\textbf{Case 1: $d\geq 1$.} By Lemma~\ref{lem:simultaneous-local-reduction}(i), after shrinking
around $A$ we have a product decomposition $U\simeq U_0\times\Sym^2\C^{\,n-d,\vee}$
such that $I_r^{(n)}\OO_U= I_{r-d}^{(n-d)}\OO_U.$ Moreover,  $Y_j^{(n)}\cap U=\varnothing$ for $j\leq d,$ whereas $Y_j^{(n)}\cap U=U_0\times Y_{j-d}^{(n-d)},$  for $j>d$. Hence the blow-ups with centers $Y_1^{(n)},\ldots,Y_d^{(n)}$ do not meet $U$, and the restriction to $U$ of the remaining Gaube--Schober tower is the product with $U_0$ of the Gaube--Schober tower $Y_1^{(n-d)},\ldots,Y_{r-d-1}^{(n-d)}$ for the residual symmetric matrix due to the fact that blowing up commutes with this product decomposition.  We may therefore apply the induction hypothesis to the residual
ideal $I_{r-d}^{(n-d)}$. It gives
$$
I_{r-d}^{(n-d)}\OO= \OO \left( -\sum_{\ell=1}^{r-d-1}(r-d-\ell+1)\overline E_\ell \right)I_{\overline Y_{r-d}},
$$
where $\overline E_\ell$ is the exceptional divisor corresponding
to the residual center $Y_\ell^{(n-d)}$. Under the identification $\ell=j-d,$ the divisor $\overline E_\ell$ is precisely $E_j$, and $r-d-\ell+1 =r-j+1.$ Thus~\eqref{eq:factorizing-GS} follows in a neighbourhood of $\xi$.

\medskip

\noindent
\textbf{Case 2: $A=0$.} Let $\xi_1\in X_1$ be the image of $\xi$ under the morphism $X_{r-1}\longrightarrow X_1.$
By Lemma~\ref{lem:simultaneous-local-reduction}(ii), after a constant change of basis and shrinking around $\xi_1$, there are coordinates in which
\begin{equation}
\label{eq:first-exceptional-factor}
I_r^{(n)}\OO_{X_1}
=
t^r I_{r-1}^{(n-1)}(B),
\end{equation}
where $E_1=(t=0).$ Moreover, the strict transforms of $Y_2^{(n)},\ldots,Y_{r-1}^{(n)}$ are on this neighbourhood exactly the products of a smooth factor with $Y_1^{(n-1)},\ldots,Y_{r-2}^{(n-1)}.$ Hence the remaining blow-ups are the product of a smooth factor with the Gaube--Schober tower for the residual pair $(n-1,r-1)$. Applying the induction hypothesis to $I_{r-1}^{(n-1)}(B)$ gives
$$
I_{r-1}^{(n-1)}(B)\OO=\OO\left(-\sum_{\ell=1}^{r-2}((r-1)-\ell+1)\overline E_\ell\right)I_{\overline Y_{r-1}}.
$$
The residual divisor $\overline E_\ell$ corresponds to the original divisor $E_{\ell+1},$ and $(r-1)-\ell+1=r-(\ell+1)+1.$ Together with the factor $t^r$ in \eqref{eq:first-exceptional-factor}, which contributes $rE_1$, this gives
$$ I_r^{(n)}\OO= \OO \left( -rE_1 -\sum_{j=2}^{r-1}(r-j+1)E_j \right) I_{Y_r^{(r-1)}}
$$
which is is exactly~\eqref{eq:factorizing-GS}.
\end{proof}

For the topological zeta function of an ideal, one in fact needs a log-principalization. Thus there is one final step.
Recall that
\[
\codim_{S_n}Y_r^{(n)}
=
\binom{n-r+2}{2}.
\]

\begin{corollary}
\label{cor:GS-principalization}
Let $\pi^{(r-1)}:X_{r-1}\longrightarrow S_n$ be the Gaube--Schober embedded resolution above. 
\begin{itemize}
\item If $\codim_{S_n}Y_r^{(n)}>1,$ let $\pi_r: X_r = \Bl_{Y_r^{(r-1)}}(X_{r-1}) \longrightarrow X_{r-1}$ be the blow-up of the smooth strict transform $Y_r^{(r-1)}$, and let $E_r$ denote its exceptional divisor. 
\item If $\codim_{S_n}Y_r^{(n)}=1,$ or equivalently $r=n$, put $X_r:=X_{r-1},$ and $ E_r:=Y_r^{(r-1)}.$ 
\end{itemize}
Then the composite $ \pi:X_r\longrightarrow S_n$ is a log-principalization of $I_r^{(n)}$, and
\begin{equation}
\label{eq:total-transform-principalization}
I_r^{(n)}\OO_{X_r} = \OO_{X_r} \left(-\sum_{j=1}^{r}(r-j+1)E_j \right).
\end{equation}
\end{corollary}

\begin{proof}
By Proposition~\ref{prop:factorizing-GS}, one has 
\begin{equation} \label{equation:factorizationtotal}
I_r^{(n)}\OO_{X_{r-1}}=\OO_{X_{r-1}}(-D)\,I_{Y_r^{(r-1)}},
\end{equation}
where $D:=\sum_{j=1}^{r-1}(r-j+1)E_j.$ Suppose first that $ c:=\codim_{X_{r-1}}Y_r^{(r-1)}>1.$
Since $Y_r^{(r-1)}$ is smooth, the defining property of the blow-up gives $I_{Y_r^{(r-1)}}\OO_{X_r}= \OO_{X_r}(-E_r).$
Furthermore, $Y_r^{(r-1)}$ is not contained in any $E_j$. Indeed, its generic point lies over the smooth rank-$(r-1)$ stratum $Y_r^{(n)}\setminus Y_{r-1}^{(n)},$ over which all preceding blow-ups are isomorphisms. Hence the order
of $E_j$ along the center $Y_r^{(r-1)}$ is zero. Writing again $E_j$
for its strict transform on $X_r$, we therefore have
$$\pi_r^*D=\sum_{j=1}^{r-1}(r-j+1)E_j.$$
Pulling (\ref{equation:factorizationtotal}) to $X_r$ yields (\ref{eq:total-transform-principalization}).

By \cite[Theorem B]{GS24}, the strict transform $Y_r^{(r-1)}$ is smooth and has normal crossings with the existing exceptional divisor. Hence $E_1+\cdots+E_r$ is still a simple normal crossings divisor. All centers of the construction lie over $Y_r^{(n)}$, so the composite morphism $\pi:X_r\longrightarrow S_n$ is an isomorphism over $S_n\setminus Y_r^{(n)}.$
Thus $\pi$ is a log-principalization of $I_r^{(n)}$.

Suppose finally that $c=1.$ Then necessarily $r=n$, and $Y_r^{(r-1)}$ is already a smooth
divisor. Thus, the claims follow trivially.
\end{proof}

\begin{lemma}\label{lem:order}
For every $1\le j\le r$, $N_j=r-j+1.$
\end{lemma}

\begin{proof} It follows directly from Corollary \ref{cor:GS-principalization}. 
\end{proof}

\begin{lemma}\label{lem:nu}
For $1\le j\le r$,
$$\nu_j=\binom{n-j+2}{2}.$$

\end{lemma}

\begin{proof}
At the moment when $Y_j^{(j-1)}$ is blown up, its generic point lies over $Y_j\setminus Y_{j-1}.$ The previous blow-ups are isomorphisms over this open set. In particular, the generic point of $Y_j^{(j-1)}$ is not contained in any preceding exceptional divisor, and its codimension in $X_{j-1}$ equals the codimension of $Y_j$ in $S_0$:
\[
c_j=\binom{n-j+2}{2}.
\]
Blowing up a smooth center of codimension $c_j$ contributes discrepancy $c_j-1$. Since the center is not contained in any preceding exceptional divisor, there is no additional contribution to the coefficient of the newly created divisor. Thus $
\nu_j-1=c_j-1,$ and so the claim follows.
\end{proof}
\begin{remark}
For $k=1$ and $j=r=n$, the divisor $E_r$ is the strict transform of the hypersurface $Y_n$, so $N_r=1$ and $\nu_r=1$, which agrees with the same formula:
\[
\binom{n-r+2}{2}=\binom22=1.
\]
\end{remark}
Now, combining Lemmas~\ref{lem:order} and~\ref{lem:nu}, one gets
\[
N_js+\nu_j
=
(r-j+1)s+\binom{n-j+2}{2}
=
D_j(s).
\]
As a sequence, for a log-principalization $\pi:X_r\to S_0$, the local topological zeta function of $I_k$ at $0$ is
\begin{equation}\label{eq:DL}
Z^{\mathrm{top}}_{(S_0,S_k),0}(s)
=
\sum_{\varnothing\ne J\subseteq\{1,\dots,r\}}
\chi\!\left(E_J^\circ\cap\pi^{-1}(0)\right)
\prod_{j\in J}\frac{1}{D_j(s)},
\end{equation}
where
\[
E_J^\circ
=
\left(\bigcap_{j\in J}E_j\right)
\setminus
\left(\bigcup_{\ell\notin J}E_\ell\right).
\]
So, the rest of the proof consists in computing the Euler characteristics $\chi(E_J^\circ\cap\pi^{-1}(0)).
$.

\section{Fiber over the origin and boundary strata}

\begin{lemma}\label{lem:fiber}
Let $F:=\pi^{-1}(0).$ Then $F$ is the final strict transform of the exceptional divisor of the first blow-up of $Y_1=\{0\}$. In particular, $F=E_1$
and 
\[
E_J^\circ\cap F=\varnothing
\qquad\text{unless }1\in J.
\]
\end{lemma}

\begin{proof}
After the first blow-up, $\pi_1:X_1=\Bl_0(S_0)\longrightarrow S_0,$ the fiber over $0$ is $E_1^{(1)}=\PP(S_0)=\PP(\Sym^2V^\vee).$  The second center is $Y^{(1)}_2=\overline{\pi_1^{-1}(Y_2\setminus Y_1)}$
Since $Y_2^{(1)}$ has a dense open subset mapping isomorphically to $ Y_2\setminus Y_{1}$,  $Y_2^{(1)}$ is not contained in $E_1^{(1)}$. However, $Y_2^{(1)} \cap E_1^{(1)}=\PP(Y_2)$. Hence, in the second blow up $\pi_2: X_2 = \Bl_{Y_2^{(1)}} (X_1)\rightarrow X_1,$ one has $$(\pi_1\circ \pi_2)^{-1}(0)=\pi^{-1}_2(E^{(1)}_1)=\overline{\pi^{-1}_2(E^{(1)}_1 \setminus \PP(Y_2))}=:E^{(2)}_1.$$ Inductively, the full inverse image of $0$ is the final strict transform of the first exceptional divisor. This is $E_1$. The last claim is now obvious.
\end{proof}

\begin{remark}
Inside $F$, the successive restrictions of the Gaube--Schober blow-ups are precisely the successive blow-ups of $\PP(\Sym^2V^\vee)$ along the strict transforms of the projective symmetric rank loci. This is the truncated complete-quadrics construction relevant to the fixed determinantal ideal $I_r$.
\end{remark}

Fix $ J=\{1=j_0<j_1<\cdots<j_t=b\}\subseteq\{1,\dots,r\}$ and  put
$$ d_\ell:=j_\ell-j_{\ell-1}\;(1\le\ell\le t) \text{ and } m:=n-b+1.
$$
Then $d_1+\cdots+d_t=b-1$. Define the kernel flag variety
$$\cF_J := \left\{ V=K_0\supset K_1\supset\cdots\supset K_t: \dim(K_{\ell-1}/K_\ell)=d_\ell\right\}.
$$
On $\cF_J$, let $V\otimes\OO_{\cF_J}= \cK_0\supset\cK_1\supset\cdots\supset\cK_t$ denote the tautological subbundles and set $ \cG_\ell:=\cK_{\ell-1}/\cK_\ell.$ Then $ \rank\cG_\ell=d_\ell,$ and $\rank\cK_t=m.$ For $d\ge1$, define
$$Q_d:=\left\{[q]\in\PP(\Sym^2\CC^{d\,\vee}):\rank q=d\right\}.
$$
For $m\ge1$ and $a\ge1$, define
$$
U_{m,a} :=\left\{[q]\in\PP(\Sym^2\CC^{m\,\vee}):\rank q\ge a\right\}.
$$

\begin{proposition}\label{prop:strata}
The stratum $E_J^\circ\cap F$ is a Zariski locally trivial fibration over $\cF_J$ with fiber
\begin{equation}\label{eq:fiberproduct}
Q_{d_1}\times\cdots\times Q_{d_t}\times U_{m,r-b+1}.
\end{equation}
\end{proposition}
\begin{proof}
We shall repeatedly use Lemma~\ref{lem:simultaneous-local-reduction}. The important consequence of that lemma is the following: after extracting a constant-rank block of rank $d$, not only the ambient symmetric matrix but the whole remaining chain of rank loci and blow-up centers is identified, locally and up to a smooth factor, with the corresponding chain for a residual symmetric matrix on the kernel. More precisely, if $q$ has rank $d$ and kernel $K$, then locally near $q$,
$$
Y_j \quad\longleftrightarrow\quad \left\{ B\in\Sym^2K^\vee : \rank B\leq j-d-1 \right\},
$$
and the subsequent blow-ups are the corresponding residual Gaube--Schober blow-ups. We shall refer to this as the
\emph{residual compatibility} of the blow-up tower. The proof is divided in several steps.

\medskip

\noindent
\textbf{Step 1: the first rank datum.}
A point of the first exceptional divisor is a projective symmetric form $[q_1]\in \mathbb{P}(\Sym^2V^\vee).$
After the first blow-up, one has $$
E_1\cap Y_j^{(1)}=\PP(Y_j)= \{[q]\in\PP(\Sym^2V^\vee):\rank q\leq j-1\}.
$$
Let $J=\{1=j_0<j_1<\cdots<j_t=b\},$ and let $x\in E_J^\circ\cap F.$
Since $j_1$ is the first index after $1$ belonging to $J$, the point $x$ lies on $E_{j_1}$ and avoids $ E_2,\ldots,E_{j_1-1}.$ By Lemma~\ref{lem:simultaneous-local-reduction}, the successive centers $
Y_2^{(1)},Y_3^{(2)},\ldots$ restricted to $E_1$ are exactly the successive strict transforms of
the projective rank loci $\{\rank q\leq 1\}, \{\rank q\leq 2\}, \ldots.$ Consequently, the first center in the blow-up sequence containing the projective direction $[q_1]$ is the one indexed by $\rank(q_1)+1.$ Since the first exceptional divisor after $E_1$ containing $x$ is $E_{j_1}$, it follows that $\rank(q_1)=j_1-1=d_1.$ Denote $K_1:=\ker q_1.$
Then $\dim(V/K_1)=d_1,$ and $q_1$ induces a nondegenerate projective quadratic form $[\overline q_1]
\in \PP\bigl(\Sym^2(V/K_1)^\vee\bigr)_{\mathrm{nd}}.$
Thus the first datum is
$$
K_1\subset V, \; [\overline q_1] \in\PP\bigl(\Sym^2(V/K_1)^\vee\bigr)_{\mathrm{nd}}.
$$

\medskip

\noindent
\textbf{Step 2: the normal slice.}
Fix a form $q$ of rank $d$ with kernel $K$. Choose a splitting $V=L\oplus K$ such that $q|_L$ is nondegenerate. A nearby symmetric matrix can be written as
$$
M=
\begin{pmatrix}
A&B\\
B^t&C
\end{pmatrix},
$$
with $A$ symmetric and invertible. The Schur-complement congruence gives $$
M\sim
\begin{pmatrix}
A&0\\
0&C-B^tA^{-1}B
\end{pmatrix}.
$$
Thus the restriction map to $K$ identifies the normal space to the
rank-$d$ stratum with
$$
N_{\{\rank=d\}/\Sym^2V^\vee,q}\simeq\Sym^2K^\vee.
$$
At the level of the corresponding projective rank stratum, the normal bundle may differ from $\Sym^2K^\vee$ by a line-bundle factor coming from the tautological projective line. This causes no ambiguity here, since projectivization is unchanged by tensoring with a line bundle. Hence the exceptional directions are canonically represented by projective symmetric forms on $K$. More importantly, Lemma~\ref{lem:simultaneous-local-reduction} shows that this is not merely a tangent-space description: the whole nested family of subsequent rank loci is, locally, the corresponding nested family of
symmetric rank loci in $\Sym^2K^\vee$.

\medskip

\noindent
\textbf{Step 3: iteration.} Coming back to the form $q_1$ of rank $d_1=j_1-1$ with kernel $K_1$, by Step~2 and the residual compatibility furnished by Lemma~\ref{lem:simultaneous-local-reduction}, after the blow-up producing $E_{j_1}$, the remaining local blow-up problem is, up to a smooth factor, the Gaube--Schober problem for the residual symmetric
space $\Sym^2K_1^\vee.$ The exceptional direction created at this stage is therefore represented by $[q_2]\in\PP(\Sym^2K_1^\vee).$ The divisors $E_{j_1+1},\ldots,E_{j_2-1}$ correspond, under the residual identification, to the successive residual rank loci of ranks $\leq 1,\ldots,\leq d_2-1,$ where $d_2=j_2-j_1.$ Since the point lies on $E_{j_2}$ and avoids all these intermediate divisors, the same argument as in Step~1 gives $\rank(q_2)=d_2.$
Set $K_2:=\ker q_2\subset K_1.$ Then $\dim(K_1/K_2)=d_2,$ and $q_2$ induces a nondegenerate projective quadratic form on
$K_1/K_2.$

Repeating this argument recursively, we obtain a flag $V=K_0\supset K_1\supset\cdots\supset K_t$
such that $\dim(K_{\ell-1}/K_\ell)=d_\ell,$ together with nondegenerate projective symmetric forms
\[
[q_\ell]
\in
\PP\bigl(
\Sym^2(K_{\ell-1}/K_\ell)^\vee
\bigr)_{\mathrm{nd}},
\;
1\leq \ell\leq t.
\]
Since $d_1+\cdots+d_t=b-1,$ we have $\dim K_t=n-b+1=m.$

\medskip

\noindent
\textbf{Step 4: the residual factor.}
By Lemma~\ref{lem:simultaneous-local-reduction}, the remaining local problem is, up to a smooth factor, the symmetric determinantal problem in $\Sym^2K_t^\vee,$ and the original ideal $I_r$ becomes the residual ideal $I_{r-b+1}$
of the residual symmetric matrix. At the stage corresponding to the last selected divisor $E_b$, the
residual projective datum is therefore represented by $[q_\infty]\in\PP(\Sym^2K_t^\vee).$
The later divisors $E_{b+1},\ldots,E_r$ are, by the same residual compatibility, exactly the boundary divisors
arising from the successive residual rank loci $\rank q_\infty\leq 1,\ldots, \rank q_\infty\leq r-b.$
Since $b=\max J$, the point lies on none of these later divisors. Hence
$$
\rank(q_\infty)\geq r-b+1,
$$
that is, $[q_\infty]\in U_{m,r-b+1}.$

When $b=r$ and $k>1$, the last divisor $E_r$ is created by the final blow-up of the smooth strict transform of $Y_r$. In residual coordinates the center is the origin in $\Sym^2K_t^\vee,$ and the exceptional direction is an arbitrary nonzero projective quadratic form; this is precisely $U_{m,1}.$

When $k=1$, the final strict transform is already a divisor. In this
case $b=r=n$, hence $m=1$, and $U_{1,1}$ consists of a single point, which agrees with the same description.

\medskip

\noindent
\textbf{Step 5: the bijection.} Denote $G_\ell:=K_{\ell-1}/K_\ell$ and define  $$
\mathcal Q_\ell= \left\{ [q]\in\PP_{F_J}(\Sym^2G_\ell^\vee): q\text{ is nondegenerate} \right\},$$
and 
$$
\mathcal U = \left\{ [q]\in \PP_{F_J}(\Sym^2K_t^\vee): \rank q\geq r-b+1\right\}.
$$
Set $\mathcal Z_J = \mathcal Q_1 \times_{\cF_J}\cdots\times_{\cF_J} \mathcal Q_t \times_{F_J}\mathcal U.$ By Steps~1--4, a point of $E_J^\circ\cap F$ determines intrinsically:
\begin{enumerate}
\item the successive kernel flag $ V=K_0\supset K_1\supset\cdots\supset K_t;$

\item nondegenerate projective forms $ [q_\ell]\in \PP(\Sym^2G_\ell^\vee)_{\mathrm{nd}},\;1\leq\ell\leq t;$

\item a residual projective form $[q_\infty]\in \PP(\Sym^2K_t^\vee),\;\rank(q_\infty)\geq r-b+1.$
\end{enumerate}
Thus we obtain a map $\phi:E_J^\circ\cap F\longrightarrow \mathcal Z_J.$

Conversely, suppose that $(K_\bullet,[q_1],\ldots,[q_t],[q_\infty])\in Z_J.$ The first form $[q_1]$ is a rank-$d_1$ projective form on $V$ with kernel $K_1$, and hence determines the corresponding point of the first exceptional divisor.
By Lemma~\ref{lem:simultaneous-local-reduction}, in a neighbourhood of this point the remaining blow-up tower is the residual Gaube--Schober tower in $\Sym^2K_1^\vee.$ The projective form $[q_2]$, of rank $d_2$, therefore determines
uniquely the next exceptional direction. Its kernel is $K_2$, and the residual problem becomes the corresponding one in
$\Sym^2K_2^\vee.$ Iterating, $[q_\ell]$ determines uniquely the exceptional direction at the $\ell$-th selected stage, and its exact rank $d_\ell$ guarantees that all intermediate boundary divisors are avoided. After the last selected stage, the form $[q_\infty]$ determines the residual point. The condition $\rank(q_\infty)\geq r-b+1$ is exactly the condition that none of the later divisors $E_{b+1},\ldots,E_r$ is met.

Thus, the data reconstruct a unique point of $E_J^\circ\cap F.$ This gives an inverse map $\psi:\mathcal Z_J\longrightarrow E_J^\circ\cap F.$ By construction, $ \psi\circ\phi=\id_{E_J^\circ\cap F},$ and $\phi\circ\psi=\id_{Z_J}.$

\medskip

\noindent
\textbf{Step 6: algebraicity and local triviality.}
It remains to verify that the preceding correspondence is algebraic. On every constant-rank stratum, the kernel of the universal symmetric map $V\otimes\OO\longrightarrow V^\vee\otimes\OO$ is an algebraic vector subbundle. Therefore the successive kernels $K_1,\ldots,K_t$ and quotient bundles $G_\ell=K_{\ell-1}/K_\ell$ vary algebraically.

More importantly, Lemma~\ref{lem:simultaneous-local-reduction} gives, on each of the corresponding Zariski charts, an algebraic product description of the entire residual blow-up tower. In these coordinates the maps $\phi$ and $\psi$ above are exactly the maps recording, respectively reconstructing, the successive projective exceptional coordinates. Hence both maps are regular on these charts. Since their construction uses only the  kernel flags and induced quotient forms, the local maps agree on overlaps. Consequently they glue to an algebraic isomorphism
\begin{equation}
\label{eq:stratum-isomorphism}
E_J^\circ\cap F \simeq \cZ_J.
\end{equation}

Finally, choose a Zariski open subset $W\subset \cF_J$ on which all vector bundles $G_1,\ldots,G_t,K_t$
are trivial. Then $\mathcal Q_\ell|_W\simeq W\times Q_{d_\ell},$ and $\mathcal U|_W\simeq W\times U_{m,r-b+1}.$ Using~\eqref{eq:stratum-isomorphism}, we obtain
$$(E_J^\circ\cap F)|_W \simeq W\times Q_{d_1}\times\cdots\times Q_{d_t} \times U_{m,r-b+1}.
$$
Hence, $E_J^\circ\cap F\longrightarrow \cF_J$ is a Zariski locally trivial fibration with fiber $Q_{d_1}\times\cdots\times Q_{d_t} \times U_{m,r-b+1}.$
\end{proof}

\section{Euler characteristic of boundary strata}
\begin{lemma}\label{lem:flagchi}
With the notation above,
\begin{equation}\label{eq:flagchi}
\chi(\cF_J)
=
\frac{n!}{m!\,d_1!\cdots d_t!}.
\end{equation}
\end{lemma}

\begin{proof}
The flag variety $\cF_J$ is the partial flag variety whose successive quotient dimensions are
$d_1,\dots,d_t,m.$ Its Schubert-cell decomposition is indexed by cosets of $\mathfrak{S}_{d_1}\times\cdots\times \mathfrak{S}_{d_t}\times \mathfrak{S}_m$ in $\mathfrak{S}_n$ (cf. \cite[$\S$1.2]{Bri05}) where $\mathfrak{S}_p$ is the symmetric group on $p$ letters. Hence, the formula follows.
\end{proof}
\begin{lemma}\label{lem:Qchi}
For $d\ge1$,
\begin{equation}\label{eq:Qchi}
\chi(Q_d)
=
\begin{cases}
1,&d=1,\\
1,&d=2,\\
0,&d\ge3.
\end{cases}
\end{equation}
\end{lemma}

\begin{proof}
Let $T=(\CC^\times)^d$ act diagonally on $\CC^d$. This induces an action on $\PP(\Sym^2\CC^{d\,\vee}).$
The $T$-weights of $\Sym^2\CC^{d\,\vee}$ are
$$2e_i\quad\text{and}\quad e_i+e_j\quad(i<j),
$$ where $e_1,\ldots,e_d$ are standard basis of the character lattice of $T$ which is $\mathbb
Z^d$. Hence the $T$-fixed points in the projective space are exactly the coordinate lines represented by
$$ x_i^2 \quad\text{and}\quad x_ix_j\ (i<j).
$$
The form $x_i^2$ has rank $1$, whereas $x_ix_j$ has rank $2$. Therefore the nondegenerate locus $Q_d$ has:
$$
1\text{ fixed point if }d=1,\qquad
1\text{ fixed point if }d=2,\qquad
0\text{ fixed points if }d\ge3.
$$
It is well-known that for a complex algebraic variety with an algebraic torus action, $\chi(X)=\chi(X^T).$
Thus \eqref{eq:Qchi} follows.
\end{proof}

Define
$$
\mathcal R_{m,q}
:=
\{[A]\in\PP(\Sym^2\CC^{m\,\vee}):\rank A=q\}.
$$

\begin{lemma}\label{lem:rankchi}
For $1\le q\le m$,
$$
\chi(\mathcal R_{m,q})=\binom mq\chi(Q_q).
$$
Consequently,
\begin{equation}\label{eq:Uchi}
\chi(U_{m,a}) =
\begin{cases}
\displaystyle\binom{m+1}{2},&a=1,\\[2mm]
\displaystyle\binom m2,&a=2,\\[2mm]
0,&a\ge3.
\end{cases}
\end{equation}
\end{lemma}

\begin{proof}
A rank-$q$ symmetric form is determined by its kernel $K\in\Gr(m-q,m)$ and by a nondegenerate projective symmetric form on $\CC^m/K$. Thus $\mathcal R_{m,q}\longrightarrow \Gr(m-q,m)$ is a Zariski locally trivial fibration with fiber $Q_q$. Therefore
$$
\chi(\mathcal R_{m,q})=\chi(\Gr(m-q,m))\chi(Q_q)=\binom mq\chi(Q_q).
$$
Using Lemma~\ref{lem:Qchi},
$$
\chi(\mathcal R_{m,q})
=
\begin{cases}
m,&q=1,\\[1mm]
\binom m2,&q=2,\\[1mm]
0,&q\ge3.
\end{cases}
$$
Since $U_{m,a} = \bigsqcup_{q=a}^m\mathcal R_{m,q},$ additivity of Euler characteristic gives
\[
\chi(U_{m,a})
=
\sum_{q=a}^m\chi(\mathcal R_{m,q}),
\]
which is exactly \eqref{eq:Uchi}.
\end{proof}

\section{Proof of the main symmetric results}

\begin{proof}[Proof of Theorem \ref{thm:main}] 
By Proposition~\ref{prop:strata} and multiplicativity of Euler characteristic for Zariski locally trivial fibrations,
\begin{equation}\label{eq:stratumchi-general}
\chi(E_J^\circ\cap F)
=
\frac{n!}{m!\prod_{\ell=1}^td_\ell!}
\left(\prod_{\ell=1}^t\chi(Q_{d_\ell})\right)
\chi(U_{m,r-b+1}).
\end{equation}
On one hand, by Lemma~\ref{lem:Qchi}, $(\ref{eq:stratumchi-general})$ vanishes unless $
d_\ell\in\{1,2\} $ for every $\ell.$ On the other hand, Lemma~\ref{lem:rankchi}, implies that $(\ref{eq:stratumchi-general})$ vanishes unless $r-b+1\le2.$ Hence the only possibilities are $
b=r$ or $b=r-1$. Thus the nonzero strata are exactly those indexed by paths
$$
1=j_0<j_1<\cdots<j_t=b
$$
whose jumps are $1$ or $2$, and whose endpoint is $r$ or $r-1$.

Let $q(\mathbf j)=\#\{\ell:d_\ell=2\}.$ Then one has
$$
\prod_{\ell=1}^td_\ell!
=
2^{q(\mathbf j)}.
$$ since every $d_\ell$ is $1$ or $2$.

\subsection*{The case $b=r$}

Note that $m=n-r+1=k$ and $r-b+1=1.$ Hence
$$
\chi(U_{k,1})=\binom{k+1}{2}.
$$
Therefore
\begin{equation}\label{eq:chi-br}
\chi(E_J^\circ\cap F)
=
\binom{k+1}{2}\frac{n!}{k!}\,2^{-q(\mathbf j)}.
\end{equation}

\subsection*{The case $b=r-1$}

Likewise, $m=n-r+2=k+1$ and $r-b+1=2.$ Hence $\chi(U_{k+1,2})=\binom{k+1}{2}.$
Therefore
\begin{equation}\label{eq:chi-brminus}
\chi(E_J^\circ\cap F)
=
\binom{k+1}{2}\frac{n!}{(k+1)!}\,2^{-q(\mathbf j)}.
\end{equation}
All remaining strata have Euler characteristic zero.

Substitute \eqref{eq:chi-br} and \eqref{eq:chi-brminus} into the Denef--Loeser formula \eqref{eq:DL}. The terms with endpoint $b=r$ give
$$
\binom{k+1}{2}\frac{n!}{k!}
\sum_{\mathbf j\in\mathcal P_r}
\frac{2^{-q(\mathbf j)}}{\prod_{\ell=0}^{t}D_{j_\ell}(s)}.
$$
The terms with endpoint $b=r-1$ give
$$
\binom{k+1}{2}\frac{n!}{(k+1)!}
\sum_{\mathbf j\in\mathcal P_{r-1}}
\frac{2^{-q(\mathbf j)}}{\prod_{\ell=0}^{t}D_{j_\ell}(s)}.
$$
This proves \eqref{eq:pathformula}.

Define
$$
F_b(s)
:=
\sum_{\mathbf j\in\mathcal P_b}
\frac{2^{-q(\mathbf j)}}{\prod_{j\in\mathbf j}D_j(s)}.
$$
Every path ending at $b$ reaches $b$ either:
\begin{itemize}
\item from $b-1$ by a jump of length $1$, or
\item from $b-2$ by a jump of length $2$, which contributes the extra factor $1/2$.
\end{itemize}
Hence
\begin{equation}\label{eq:Frec}
F_b(s)
=
\frac{1}{D_b(s)}
\left(
F_{b-1}(s)+\frac12F_{b-2}(s)
\right).
\end{equation}
Also
$$
F_1(s)=\frac1{D_1(s)}.
$$

Write
$$
F_b(s)
=
\frac{R_b(s)}{D_1(s)\cdots D_b(s)}.
$$
Then \eqref{eq:Frec} becomes
$$
R_b(s)
=
R_{b-1}(s)+\frac{D_{b-1}(s)}2R_{b-2}(s),
$$
with $ R_0=0,$ and $ R_1=1.$
Thus $R_b$ is precisely the polynomial defined by \eqref{eq:Rinit}--\eqref{eq:Rrec}.

Therefore
$$
F_r(s)
=
\frac{R_r(s)}{\prod_{j=1}^rD_j(s)}
$$
and
$$
F_{r-1}(s)
=
\frac{R_{r-1}(s)}{\prod_{j=1}^{r-1}D_j(s)}
=
\frac{D_r(s)R_{r-1}(s)}{\prod_{j=1}^rD_j(s)}.
$$
Using
$$
\frac{n!}{(k+1)!}
=
\frac1{k+1}\frac{n!}{k!},
$$
the path formula becomes
$$
Z^{\mathrm{top}}_{(S_0,S_k),0}(s)
=
\binom{k+1}{2}\frac{n!}{k!}
\frac{
R_r(s)+\dfrac{D_r(s)}{k+1}R_{r-1}(s)
}{
\prod_{j=1}^rD_j(s)
}.
$$
This is exactly \eqref{eq:main}, completing the proof of Theorem~\ref{thm:main}. 
\end{proof}

The recurrence for $R_r(s)$ in fact admits an explicit expansion.

\begin{proposition}\label{prop:Rexplicit}
For $r\ge1$,
$$
R_r(s)
=
\sum_{\substack{
M\subseteq\{2,\dots,r-1\}\\
\text{no two elements of }M\text{ are consecutive}
}}
\frac1{2^{|M|}}
\prod_{j\in M}D_j(s).
$$
\end{proposition}

\begin{proof} By construction,  \begin{align*}
R_r(s) &=F_r(s)D_1(s)\cdots D_r(s) =\sum_{\mathbf j\in\mathcal P_r}2^{-q(\mathbf j)}\frac{D_1(s)\cdots D_r(s)}{\prod_{j\in\mathbf j}D_j(s)} 
\end{align*}
Consider paths $1=j_0<j_1<\cdots<j_t=r$  whose successive jumps satisfy $ j_\ell-j_{\ell-1}\in\{1,2\}.$ Thus one moves from $1$ to $r$ by steps of length $1$ or $2$. To such a path, associate the set of skipped indices $$ M(\mathbf j) := \{2,\ldots,r-1\}\setminus\{j_0,\ldots,j_t\}. $$ An index is skipped precisely when the path makes a jump of length $2$. More precisely, if $j_{\ell}-j_{\ell-1}=2,$ then the unique skipped index is $ j_{\ell-1}+1.$ Two skipped indices can never be consecutive. Indeed, suppose that $a$ is skipped. Then the path must contain the jump $a-1\longrightarrow a+1.$ In particular, $a+1$ belongs to the path and hence cannot also be skipped. Therefore $M(\mathbf j)\subseteq\{2,\ldots,r-1\}$ contains no two consecutive integers. 

Conversely, let $$M\subseteq\{2,\ldots,r-1\} $$ be a subset containing no two consecutive integers. Then there is a unique path $$ 1=j_0<j_1<\cdots<j_t=r, \qquad j_\ell-j_{\ell-1}\in\{1,2\}, $$ whose skipped indices are exactly the elements of $M$. Namely, whenever $a\in M$, the path jumps from $a-1$ to $a+1$, while at every other stage it moves by a step of length $1$. The condition that $M$ contains no consecutive integers guarantees that these jumps are compatible. Hence we obtain a bijection $$ \left\{ \begin{array}{c} \text{paths }1=j_0<\cdots<j_t=r\\ \text{with }j_\ell-j_{\ell-1}\in\{1,2\} \end{array} \right\} \longleftrightarrow \left\{ \begin{array}{c} M\subseteq\{2,\ldots,r-1\}\\ \text{with no two consecutive elements} \end{array} \right\}. $$ In particular, if $q(\mathbf j)$ denotes the number of jumps of length $2$, then $ q(\mathbf j)=|M(\mathbf j)|.$  Moreover,  every skipped vertex $j \in M(\mathbf{j})$ contributes a factor $D_j$ surviving after cancellation in the second factor $$\frac{D_1(s)\cdots D_r(s)}{\prod_{j\in\mathbf j}D_j(s)}.$$ Summing over all such subsets gives the desired expression.
\end{proof}

\begin{proof}[Proof of Corollary \ref{cor:main}]
For $k=1$, we have $r=n$.  From now on we use notations $D_j^{(n)}$ and $R_j^{(n)}$ to record the fact that the polynomials $D_j$ and $R_j$ depend on $n$. Since
$$
D_j^{(n)}(s)
=
(n-j+1)
\left(
s+\frac{n-j+2}{2}
\right),
$$
putting $
i=n-j+1$
gives
$$
\prod_{j=1}^nD_j^{(n)}(s)
=
n!
\prod_{i=1}^n
\left(
s+\frac{i+1}{2}
\right).
$$ 
Hence, the general formula for the topological zeta function is simplified to
\begin{equation}\label{eq:Z-R-simplified}
Z^{\mathrm{top}}_{(S_0,S_1),0}(s)
=
\frac{R_{n+1}^{(n)}(s)}
{\displaystyle
\prod_{i=1}^n
\left(
s+\frac{i+1}{2}
\right)}.
\end{equation} where 
$$
R_{n+1}^{(n)}(s)
=
R_n^{(n)}(s)
+
\frac{D_n^{(n)}(s)}{2}
R_{n-1}^{(n)}(s).
$$ obtained by the recurrence by one step further.

The problem now is reduced to computing $R_{n+1}^{(n)}(s)$ explicitly. By Proposition \ref{prop:Rexplicit}, we have
$$
R_{n+1}^{(n)}(s)
=
\sum_{\substack{
M\subseteq\{2,\ldots,n\}\\
\text{no two elements of $M$ are consecutive}
}}
\prod_{j\in M}
\frac{D_j^{(n)}(s)}{2}.
$$
Put $i=n-j+1.$
Then
$$
j\in\{2,\ldots,n\}
\qquad\Longleftrightarrow\qquad
i\in\{1,\ldots,n-1\},
$$
and
$$
D_j^{(n)}(s)
=
i\left(
s+\frac{i+1}{2}
\right).
$$
Moreover, the property that no two indices are consecutive is preserved by the change of variables $i=n-j+1$. Consequently,
\begin{equation}\label{eq:R-reindexed}
R_{n+1}^{(n)}(s)
=
\sum_{\substack{
M\subseteq\{1,\ldots,n-1\}\\
\text{no two elements of $M$ are consecutive}
}}
\prod_{i\in M}
\frac{i}{2}
\left(
s+\frac{i+1}{2}
\right).
\end{equation}
Let $M$ be such an admissible set. On one hand, if $n-1\notin M,$ then $
M\subseteq\{1,\ldots,n-2\},$ and the corresponding sum is precisely $R_n^{(n-1)}(s).$ On the other hand, if
$n-1\in M,$ then necessarily $n-2\notin M.$ After removing the element $n-1$, the remaining subset is an admissible subset of $\{1,\ldots,n-3\}.$
Hence the corresponding contribution is
$$
\frac{n-1}{2}
\left(
s+\frac n2
\right)
R_{n-1}^{(n-2)}(s).
$$
Thus we obtain the recurrence
\begin{equation}\label{eq:terminal-recurrence}
R_{n+1}^{(n)}(s)
=
R_n^{(n-1)}(s)
+
\frac{n-1}{2}
\left(
s+\frac n2
\right)
R_{n-1}^{(n-2)}(s).
\end{equation}
\textbf{Claim}: For every $n=2m+1$ with $m\geq0$,
\begin{equation}\label{eq:R-odd}
R_{2m+2}^{(2m+1)}(s)
=
\frac{(2m+1)!!}{2^m}
\prod_{\ell=2}^{m+1}(s+\ell),
\end{equation}
and, for every $n=2m$ with $m\geq1$,
\begin{equation}\label{eq:R-even}
R_{2m+1}^{(2m)}(s)
=
\frac{(2m-1)!!}{2^m}
(s+2m+1)
\prod_{\ell=2}^{m}(s+\ell).
\end{equation}
Indeed, we shall prove these two identities simultaneously by induction. For $n=1$, one has $R_2^{(1)}(s)=1,$ which agrees with \eqref{eq:R-odd} for $m=0$. For $n=2$,
$$
R_3^{(2)}(s)
=
R_2^{(2)}(s)
+
\frac{D_2^{(2)}(s)}2R_1^{(2)}(s).
$$
Since $ R_2^{(2)}=1$ and $ D_2^{(2)}(s)=s+1,$
we get
$$
R_3^{(2)}(s)
=
1+\frac{s+1}{2}
=
\frac{s+3}{2},
$$
which agrees with \eqref{eq:R-even} for $m=1$.

Assume now that \eqref{eq:R-even} holds in dimension $2m$ and that \eqref{eq:R-odd} holds in dimension $2m-1$. Applying \eqref{eq:terminal-recurrence} with $n=2m+1$, we obtain
 \begin{align*}
R_{2m+2}^{(2m+1)} &= R_{2m+1}^{(2m)}+m\left(s+m+\frac12\right)R_{2m}^{(2m-1)}\\
  &=\frac{(2m-1)!!}{2^m}(s+2m+1)\prod_{\ell=2}^{m}(s+\ell) +m\left(s+m+\frac12\right)\frac{(2m-1)!!}{2^{m-1}}
\prod_{\ell=2}^{m}(s+\ell) \\
  &=\left[\frac{(2m-1)!!}{2^m}\prod_{\ell=2}^{m}(s+\ell) \right] \left[
s+2m+1+2ms+2m^2+m
\right]\\
  &= \left[\frac{(2m-1)!!}{2^m}\prod_{\ell=2}^{m}(s+\ell) \right] (2m+1)(s+m+1)\\
  &= \frac{(2m+1)!!}{2^m}
\prod_{\ell=2}^{m+1}(s+\ell)
\end{align*}
where the second line follows by the induction hypothesis. Hence, \eqref{eq:R-odd} is proved.

Next apply \eqref{eq:terminal-recurrence} with $n=2m+2$:
$$
R_{2m+3}^{(2m+2)}
=
R_{2m+2}^{(2m+1)}
+
\frac{2m+1}{2}(s+m+1)
R_{2m+1}^{(2m)}.
$$
By the formulas already proved, $$
R_{2m+2}^{(2m+1)}
=
\frac{(2m+1)!!}{2^m}
\prod_{\ell=2}^{m+1}(s+\ell),
$$ and
$$
R_{2m+1}^{(2m)}
=
\frac{(2m-1)!!}{2^m}
(s+2m+1)
\prod_{\ell=2}^{m}(s+\ell), 
$$ the same kind of computations shows \eqref{eq:R-even} with $m$ replaced by $m+1$. This completes the induction.

Finally, substituting the formulas \eqref{eq:R-odd} and \eqref{eq:R-even} for $R_{n+1}^{(n)}$ into \eqref{eq:Z-R-simplified} gives the desired local topological zeta function. By \cite[Ex. 4.1.9]{Lo16}, the Bernstein-Sato polynomial is given by $$\prod_{i=1}^n\left( s+\frac{i+1}{2} \right).$$ The last statement is now obvious. This finishes the proof.
\end{proof}
\section{The generic catalecticant determinantal case}
\label{sec:cat}
In this section, a generic catalecticant matrix means a catalecticant of binary forms, or equivalently a generic Hankel matrix.  We allow rectangular matrices. 


\subsection{Simultaneous catalecticant reduction}

The following simultaneous reduction is based on the triangular matrix transformation established by Brogan in the proof of \cite[Lemma~3.1]{Bro26}. More precisely, the matrix identity below, in the square case, is the case $k=0$ of Brogan's construction. By considering suitable rectangular submatrices, we obtain the corresponding identity for arbitrary
generic Hankel matrices. We record explicitly its scheme-theoretic consequences for all determinantal ideals, which will be essential for the log-principalization constructed below. For completeness, we give a direct proof of the matrix identity using generating functions. 

Assume $m,n\geq2$, and denote $R:=\C[x_0,x_0^{-1},x_1,\ldots,x_d].$ Set $x_k=0$ for $k>d$, and define
$$ X(T):=\sum_{k=0}^{d}x_kT^k, \;\;\; P(T):=X(T)^{-1}=\sum_{k\geq0}p_kT^k\in R[[T]].$$ For $a\geq1$, let
$P_a:=(p_{j-i})_{0\leq i,j\leq a-1},\; p_k:=0\;(k<0),$ and $z_k:=-p_{k+2}$ for $0\leq k\leq d-2$. 

\begin{lemma}
\label{lem:cat-reciprocal} With the above settings, we have the following.
\begin{enumerate}[label=\textup{(\roman*)}]
\item The functions $p_0,p_1,z_0,\ldots,z_{d-2}$ are algebraic coordinates on $\{x_0\neq0\}$, with $p_0$ invertible.  Thus, $$
\{x_0\neq0\}\simeq\C^*\times\A^1\times\A^{d-1}.
$$
\item The matrices $P_m$ and $P_n$ are invertible over $R$, and
\begin{equation}
\label{eq:cat-block-identity}
P_m^{\mathrm t}H_{m,n}(x)P_n=
\begin{pmatrix}p_0&0\\0&H_{m-1,n-1}(z)
\end{pmatrix}.
\end{equation}
\item With the convention $I_0(B)=(1)$ for any matrix $B$, one has,
simultaneously for $1\leq q\leq\min\{m,n\}$,
\begin{equation}
\label{eq:cat-open-minors}
I_q(H_{m,n})R=I_{q-1}(H_{m-1,n-1}(z))R.
\end{equation}
\end{enumerate}
\end{lemma}

\begin{proof}
The reciprocal coefficients satisfy
\begin{equation}
\label{eq:cat-reciprocal-recursion}
p_0=x_0^{-1},
\qquad
p_k=-x_0^{-1}\sum_{i=1}^{\min\{k,d\}}x_i p_{k-i}
\quad(k\geq1).
\end{equation}
For $1\leq k\leq d$, this gives $p_k=-x_0^{-2}x_k+\text{an element of }\C[x_0,x_0^{-1},x_1,\ldots,x_{k-1}].$
 The transformation $(x_0,\ldots,x_d)\longmapsto(p_0,\ldots,p_d)$ is an algebraic isomorphism on the open set $x_0\neq 0$.
Indeed, since $X(T)=P(T)^{-1}$, its inverse is given recursively by
$$x_0=p_0^{-1},\; x_k=-p_0^{-1}\sum_{i=1}^{k}p_i x_{k-i},\;1\leq k\leq d.
$$
In particular, $\mathbb C[x_0,x_0^{-1},x_1,\ldots,x_d]=\mathbb C[p_0,p_0^{-1},p_1,\ldots,p_d].$  This proves (i).  Also $P_a$ is upper triangular with diagonal entries $p_0$, so $\det P_a=p_0^a$ is a unit.

We now prove \eqref{eq:cat-block-identity}.  Introducing two independent formal variables $u,v$ we have
\begin{equation}\label{eq:cat-hankel-generating-function}
\sum_{a,b\geq0}x_{a+b}u^av^b =\sum_{k=0}^{d}x_k\sum_{a=0}^{k}u^av^{k-a}\notag=\frac{uX(u)-vX(v)}{u-v}.
\end{equation}
so that the quotient on the right is a polynomial.  Multiplying by $P(u)P(v)$ and
using $P(T)X(T)=1$, we obtain
\begin{equation}
\label{eq:cat-transformed-generating-function}
P(u)P(v)\sum_{a,b\geq0}x_{a+b}u^av^b=\frac{uP(v)-vP(u)}{u-v}.
\end{equation}
On one hand, in the right hand side of \eqref{eq:cat-transformed-generating-function} the terms involving $p_1$ cancel, and, for $k\geq2$,
\[
\frac{uv^k-vu^k}{u-v}
=-\sum_{i=1}^{k-1}u^iv^{k-i}.
\]
Consequently, as an identity in $R[[u,v]]$,
\begin{equation}
\label{eq:cat-coefficient-identity}
\frac{uP(v)-vP(u)}{u-v}
=p_0-\sum_{i,j\geq1}p_{i+j}u^iv^j.
\end{equation}
On the other hand, for $0\leq i\leq m-1$ and $0\leq j\leq n-1$, the coefficient of $u^iv^j$ on the left hand side of \eqref{eq:cat-transformed-generating-function} is
$$
\sum_{a=0}^{i}\sum_{b=0}^{j}p_{i-a}x_{a+b}p_{j-b}
=(P_m^{\mathrm t}H_{m,n}P_n)_{ij}.
$$
Therefore, by \eqref{eq:cat-coefficient-identity}, one has

\begin{equation*}
(P_m^{\mathrm t}H_{m,n}P_n)_{ij} =
\begin{cases}
p_0,&\text{ if }i=j=0,\\[2mm]
0,&\text{ if either } i=0 \text{ and } j\geq 1 \text{ or } j=0 \text{ and } i\geq 1 ,\\[2mm]
-p_{i+j},& \text{ otherwise},
\end{cases}
\end{equation*} or equivalently $$
P_m^{\mathrm t}H_{m,n}P_n=
\begin{pmatrix}
p_0 & 0 & 0 & \cdots & 0\\
0 & -p_2 & -p_3 & \cdots & -p_n\\
0 & -p_3 & -p_4 & \cdots & -p_{n+1}\\
\vdots & \vdots & \vdots & \ddots & \vdots\\
0 & -p_m & -p_{m+1} & \cdots & -p_{m+n-2}
\end{pmatrix}=\begin{pmatrix}
p_0&0\\
0&H_{m-1,n-1}(z).
\end{pmatrix}
$$
This proves (ii) and also explains the appearance of the factor $\mathbb{A}^1$ in (i).

Finally, (iii) follows directly from (ii).
\end{proof}

For clarity, the ordinary Schur complement is not a substitute for
Lemma~\ref{lem:cat-reciprocal}.  For a $4\times4$ Hankel matrix, eliminating
the first row and column gives entries
$$B_{13}=x_4-\frac{x_1x_3}{x_0},\; B_{22}=x_4-\frac{x_2^2}{x_0},
$$
which need not agree.  The residual matrix obtained that way is not
generally Hankel.

\begin{lemma}
\label{lem:cat-coverage}
The natural $\GL_2(\C)$-action on $S_0$ preserves every ideal $I_q$.
The translates of $\{x_0\neq0\}$ cover $S_0\setminus\{0\}$, and the
translates of its projective chart cover $\PP^d$.  Likewise, the translates
of the $x_0$-direction chart cover $\Bl_0(S_0)$.
\end{lemma}

\begin{proof}
Let $W=\C^2$ with basis $e_0,e_1$.  Identify $S_0=(\Sym^d W)^\vee$ by associating to $x$ the functional $\lambda_x$
with $\lambda_x(e_0^{d-k}e_1^k)=x_k.$ The matrix $H_{m,n}(x)$ represents the multiplication pairing
$$\Sym^{m-1}W\times\Sym^{n-1}W\longrightarrow\C, \;(f,g)\longmapsto\lambda_x(fg).
$$
A change of basis of $W$ induces invertible changes of basis in the two factors.  Thus, its action on this matrix preserves all the ideals $I_q$.

For a nonzero functional $\lambda_x$, the polynomial
$$
\lambda_x((ae_0+be_1)^d)=\sum_{k=0}^{d}\binom dk x_k a^{d-k}b^k
$$
is not identically zero.  Hence some nonzero $\ell\in W$ satisfies $\lambda_x(\ell^d)\neq0$.  Choosing $\ell$ as the first basis vector places $x$ in a chart with first coordinate nonzero.  This proves the assertions on the affine and projective spaces.

The $\mathrm{GL}(W)$-action lifts to the blow-up because it preserves the ideal of the origin.  Every point of this blow-up has a projective direction $[\lambda]\in\PP((\Sym^dW)^\vee)$.  Applying the same argument to this direction places the point in an $x_0$-direction chart.  Since all the ideals are invariant, the $\mathrm{GL}(W)$-action also lifts to the subsequent blow-ups of their strict transforms.
\end{proof}

To indicate matrix sizes when necessary, we write $S_0^{m',n'}:=\A^{m'+n'-1}$ and $Y_q^{m',n'}:=V(I_q(H_{m',n'}))$.

\begin{lemma}
\label{lem:cat-local-reduction}
Assume $m,n\geq2$.
\begin{enumerate}[label=\textup{(\roman*)}]
\item Let $x\in S_0$ have rank $a$, where $1\leq a<\min\{m,n\}$. There is a Zariski neighbourhood $U$ of $x$ and an isomorphism
$$
U\simeq U_0\times S_0^{m-a,n-a},\;\;\;\dim U_0=2a,$$
with $U_0$ smooth and $x$ mapping to $(u_0,0)$, such that, for every $1\leq q\leq\min\{m,n\}$,
\begin{equation*}
\label{eq:cat-positive-rank-reduction}
I_q\cO_U=
\begin{cases}
\cO_U,&q\leq a,\\
I_{q-a}(H_{m-a,n-a})\cO_U,&q>a.
\end{cases}
\end{equation*}
\item Let $\sigma:\Bl_0(S_0)\longrightarrow S_0$.  Around every point there is a chart with coordinates $
t,\ p_1,\ z_0,\ldots,z_{d-2}$ on which the first exceptional divisor is $(t=0)$ and, simultaneously
for $1\leq q\leq\min\{m,n\}$,
\begin{equation*}
\label{eq:cat-blowup-minors}
I_q\cO=t^q I_{q-1}(H_{m-1,n-1}(z))\cO.
\end{equation*}
For $q\geq2$, the strict transform of $Y_q$ on this chart is therefore
\begin{equation*}
\label{eq:cat-strict-transform}
\A^1_t\times\A^1_{p_1}\times Y_{q-1}^{m-1,n-1}.
\end{equation*}
\item At a projective point of exact rank $1\leq a<\min\{m,n\}$, there is a
Zariski neighbourhood
\[
U\simeq U_0\times S_0^{m-a,n-a},
\qquad \dim U_0=2a-1,
\]
with $U_0$ smooth, on which the dehomogenized determinantal ideals satisfy
\eqref{eq:cat-positive-rank-reduction}.  The exact-rank-$a$ stratum in this
neighbourhood is $U_0\times\{0\}$.
\end{enumerate}
All these identifications preserve the whole chain of determinantal ideals and hence the subsequent blow-up centers. 
\end{lemma}

\begin{proof}
For (i), by Lemma~\ref{lem:cat-coverage} we can arrange $x_0\neq0$. Lemma~\ref{lem:cat-reciprocal} implies that $U:=\lbrace x_0  \neq 0\rbrace\cong \C^*\times\A^1 \times S_0^{m-1,n-1}$ and that one can replace the matrix $H_{m,n}(x)$ by
$$\begin{pmatrix}p_0&0\\0&H_{m-1,n-1}(z)
\end{pmatrix}.$$
 where $H_{m-1,n-1}(z)$ is the smaller generic Hankel matrix.  The residual matrix at the chosen point has rank $a-1$.  If this rank is positive, apply Lemma~\ref{lem:cat-coverage} to the residual binary-form space and repeat. After $a$ such steps the residual matrix has rank zero. Since each step adds two smooth coordinates, the remaining number of coordinates is $d+1-2a$.  Now, one may take $U_0\simeq(\C^*\times\A^1)^a$ and then  (i) follows from Lemma Lemma~\ref{lem:cat-reciprocal}.(iii).  

For (ii), on the $x_0$-direction chart, we write $x_0=t,\; x_k=tu_k\ (1\leq k\leq d),\; u_0=1.$ Then $\sigma^*H_{m,n}=tH_{m,n}(1,u_1,\ldots,u_d)$. So, we can apply Lemma~\ref{lem:cat-reciprocal} to the normalized matrix, whose reciprocal coefficient $p_0$ is $1$ and then the claims follows. 

For (iii), we first normalize $x_0=1$ on a projective chart.  The first reduction gives rise to one smooth coordinate, namely $p_1$, rather than two. The subsequent $a-1$ reductions are the affine reductions of (i), each giving rise to two smooth coordinates.  Thus, the smooth factor $U_0$ has dimension $1+2(a-1)=2a-1$.  The rank is $a$ plus the rank of the final residual matrix, so the exact-rank-$a$ stratum is precisely $U_0\times\{0\}$. 

Finally, at each step the same invertible matrices work for every minor size. Thus, the last claim follows.
\end{proof}

\subsection{Principalization and numerical data}
As a particular case of \cite[Corollary 2.4]{Ber92}, the determinantal variety $Y_r$ admits an embedded resolution obtained by successively blowing up the strict transforms of $Y_1,\ Y_2,\ \ldots,\ Y_{r-1}.$ More precisely, we set  $X_0:=S_0$.  For $1\leq j\leq r-1$, let $\pi_j:X_j\longrightarrow X_{j-1}$ be the blow-up of the strict transform
$Y_j^{(j-1)}$ of $Y_j$. Then $Y_r^{(r-1)}$ is smooth and meets the exceptional divisor with simple normal crossings. Let $ \pi^{(r-1)} := \pi_1\circ\cdots\circ\pi_{r-1} : X_{r-1}\longrightarrow S_n$ be the resulting morphism, and let $E_j\subset X_{r-1}$ denote the final strict transform of the exceptional divisor created by
$\pi_j$. Here, we give an alternative direct proof of this fact in the catalecticant setting, based on the simultaneous local reduction established above. This argument also provides the scheme-theoretic factorization of the determinantal ideal required for log-principalization.
\begin{proposition}
\label{prop:cat-factorization}
Every $X_j$ is smooth, the successive centers are smooth, and the strict transform $Y_r^{(r-1)}\subset X_{r-1}$ is smooth and transverse to every stratum of the exceptional divisor.  This exceptional divisor has simple normal crossings.  Moreover,
\begin{equation}
\label{eq:cat-factorization}
I_r\cO_{X_{r-1}} = \cO_{X_{r-1}}\left(-\sum_{j=1}^{r-1}(r-j+1)E_j\right) \mathcal I_{Y_r^{(r-1)}}.
\end{equation}
\end{proposition}

\begin{proof}
We argue by induction on $r$, simultaneously for all admissible matrix
sizes.  For $r=1$ there are no blow-ups: $Y_1$ is the smooth origin and
the asserted factorization is $I_1=\mathcal I_{Y_1}$.

Assume $r\geq2$.  The first center is the smooth origin.  By Lemma~\ref{lem:cat-local-reduction}(ii), $X_1$ is covered by charts $\A^1_t\times\A^1_{p_1}\times S_0^{m-1,n-1}$ on which $E_1=(t=0)$, the strict transform of each $Y_j$, $j\geq2$, is
$\A^2\times Y_{j-1}^{m-1,n-1}$, and $I_r\cO_{X_1}=t^r I_{r-1}(H_{m-1,n-1}).$ Since, blowing up a product center commutes with this product decomposition, the remaining $r-2$ blow-ups, together with the strict transform of $Y_r$, are the product with $\A^2$ of the corresponding construction for $(m-1,n-1,r-1)$.  A simple induction now proves smoothness and the normal-crossing claims for the residual boundary and target.  The old divisor $t=0$ is independent of all residual coordinates, so it is transverse to every residual boundary stratum and to the residual target.  This also proves the assertions involving $E_1$.

For the ideal decomposition, induction gives a factorization of the residual ideal with multiplicities
$(r-1)-\ell+1=r-\ell\;(1\leq\ell\leq r-2).$ The residual divisor indexed by $\ell$ is the original divisor
$E_{\ell+1}$.  Subsequent centers are independent of $t$, so pulling back
$t=0$ introduces no additional exceptional component.  Multiplication
by $t^r$ therefore yields
$$
I_r\cO = \cO\left(-rE_1-\sum_{j=2}^{r-1}(r-j+1)E_j\right) \mathcal I_{Y_r^{(r-1)}}.
$$
proving \eqref{eq:cat-factorization}.
\end{proof}

\begin{corollary}
\label{cor:cat-principalization} Let $ \pi^{(r-1)} : X_{r-1}\longrightarrow S_n$ be the Bertram embedded resolution above.
\begin{enumerate}
\item If $c>1$, blow up the smooth strict transform $Y_r^{(r-1)}$ and denote the resulting space by $X_r$ and its exceptional divisor by $E_r$. 
\item If $c=1$, put $X_r=X_{r-1}$ and $E_r=Y_r^{(r-1)}$. 
\end{enumerate}
Then the resulting morphism $\pi:X_r\longrightarrow S_0$ is a log-principalization of $I_r$, with
\begin{equation}
\label{eq:cat-total-transform}
I_r\cO_{X_r}
=\cO_{X_r}\left(-\sum_{j=1}^{r}(r-j+1)E_j\right).
\end{equation}
For $1\leq j\leq r$, its numerical data are
\begin{equation}
\label{eq:cat-numerical-data}
N_j=r-j+1,
\qquad
\nu_j=m+n-2j+1.
\end{equation}
\end{corollary}

\begin{proof} The proof is similar to those of Corollary \ref{cor:GS-principalization}, Lemma \ref{lem:order} and Lemma \ref{lem:nu}.
\end{proof}

\subsection{The fiber over the origin}

Let $F:=\pi^{-1}(0)$ and, for a nonempty $J\subseteq\{1,\ldots,r\}$, put $$E_J^\circ:=\left(\bigcap_{j\in J}E_j\right) \setminus\left(\bigcup_{j\notin J}E_j\right). $$
Let $q_1:\Bl_0(S_0)\to\PP^d$ be the natural projective-direction map, and let $\sigma:X_r\to\Bl_0(S_0)$ be the morphism obtained by composing the subsequent blow-ups. 
Finally, denote $\rho:=(q_1\circ\sigma)|_F: F\longrightarrow\PP^d.$
\begin{lemma}
\label{lem:cat-fiber}
The fiber $F$ is $E_1$, and is reduced.  In particular, $E_J^\circ\cap F=\varnothing\text{ unless }1\in J.$
Moreover, under $F=E_1$, the morphism $\rho:F\longrightarrow\PP^d$ is obtained by successively blowing up the strict transforms of $\PP(Y_2),\ldots,\PP(Y_r)$, omitting the last blow-up when $c=1$. When $c=1$ and $r\geq2$, the strict transform of $\PP(Y_r)$ is retained as a boundary divisor.
\end{lemma}

\begin{proof}
The blow-up of $\A^{d+1}$ at the origin is the total space $\Tot(\cO_{\PP^d}(-1))$.  Its zero section is the first exceptional divisor when $d\geq1$.  For $j\geq2$, the strict transform of the homogeneous cone $Y_j$ is the inverse image of $\PP(Y_j)$ under the projection $q_1:\Bl_0(S_0)\to\PP^d$. 
It follows inductively that the final ambient space is $X_r\simeq\Tot(\rho^*\cO_{\PP^d}(-1)),$ where the base is the stated projective blow-up tower.  The pullback of the maximal ideal of $0\in S_0$ is the ideal of the zero section.  This proves the claims.
\end{proof}

\subsection{Euler characteristics of the boundary strata}

For $1\leq a<r$, if we set
$$R_a:=\{[x]\in\PP^d:\rank H_{m,n}(x)=a\},\; U_r:=\{[x]\in\PP^d:\rank H_{m,n}(x)\geq r\}.
$$
then it is obvious that  $\PP^d=U_r\sqcup R_1\sqcup\cdots\sqcup R_{r-1}$.

\begin{lemma}
\label{lem:cat-rank-euler}
One has $\chi(R_a)=2\;(1\leq a<r) $ and $\chi(U_r)=c.$
\end{lemma}

\begin{proof}
Let $\C^*$ act by $\lambda\cdot x_k=\lambda^kx_k$.  On $H_{m,n}$ this is induced by left multiplication by
$\diag(1,\lambda,\ldots,\lambda^{m-1})$ and right multiplication by $\diag(1,\lambda,\ldots,\lambda^{n-1})$.  Thus all the rank strata are invariant.  The weights $0,\ldots,d$ on the ambient projective space are
distinct, so the fixed points are exactly the coordinate points
$[e_0],\ldots,[e_d]$.

At $[e_\ell]$, the matrix is supported on the anti-diagonal $i+j=\ell$. Since its nonzero entries lie in distinct rows and columns,  its rank is the number of such entries, namely
$$\rank H_{m,n}(e_\ell)=\min\{\ell+1,m,n,d-\ell+1\}.
$$
For $a<r\leq\min\{m,n\}$, exactly two coordinate points have rank $a$:$[e_{a-1}]$ and $[e_{d-a+1}]$.  These points are distinct, since $d\geq2r-2$ and $a\leq r-1$.  The torus-fixed-point identity $\chi(Z)=\chi(Z^{\C^*})$ therefore gives $\chi(R_a)=2$. By additivity, one gets $$ \chi(U_r)=\chi(\PP^d)-\sum_{a=1}^{r-1}\chi(R_a) =d+1-2(r-1)=c.$$
This includes $r=1$, when the sum is empty.
\end{proof}

We next identify the fibers of the boundary strata.  This replaces the
kernel-flag description used for the other matrix types.

\begin{proposition}
\label{prop:cat-residual-strata} (i) The restriction of $\rho$ gives an isomorphism $E_{\{1\}}^\circ\cap F\simeq U_r.$

(ii) Let $J\subseteq\{1,\ldots,r\}$ contain $1$ and satisfy $|J|\geq2$. Denote $$ a:=\min(J\setminus\{1\})-1, \qquad (m',n',r'):=(m-a,n-a,r-a), \qquad J':=\{j-a:j\in J\setminus\{1\}\}. $$ Let $\bar\pi:\bar X\to S_0^{m',n'}$ be the principalization constructed above for the residual ideal $I_{r'}(H_{m',n'})$, with divisors $\bar E_1,\ldots,\bar E_{r'}$ and fiber $\bar F=\bar\pi^{-1}(0)$. Then $1\in J'$, and
\begin{equation}
\label{eq:cat-stratum-fibration}
E_J^\circ\cap F\longrightarrow R_a
\end{equation}
is a Zariski locally trivial fibration with fiber
$\bar E_{J'}^\circ\cap\bar F$.
\end{proposition}
\begin{proof}
We first describe the restriction of the projective blow-up tower over an exact-rank stratum.  Fix an integer $1\leq b<r$ and a point $y\in R_b$.
By Lemma~\ref{lem:cat-local-reduction}(iii), there is a Zariski neighbourhood $U$ of $y$ and an algebraic isomorphism
$$
U\simeq U_0\times S_0^{m-b,n-b},\; R_b\cap U=U_0\times\{0\},
$$
where $U_0$ is smooth.  The simultaneous ideal identities in that lemma give, for every $2\leq j\leq r$,
$$
\PP(Y_j)\cap U\simeq
\begin{cases}
\varnothing,&j\leq b,\\
U_0\times Y_{j-b}^{m-b,n-b},&j>b.
\end{cases}
$$

Let $\bar\pi:\bar X\longrightarrow S_0^{m-b,n-b}$ be the principalization for the residual parameters $(m-b,n-b,r-b)$, and let $\bar E_1,\ldots,\bar E_{r-b}$ be its boundary divisors.  Since the blow-ups commute with this product, the same is true for the successive strict transforms. It follows
inductively that the restriction of the entire projective tower is an isomorphism over $U$,
\begin{equation}
\label{eq:cat-stratum-local-tower}
\rho^{-1}(U)\simeq U_0\times\bar X,\qquad\rho|_{\rho^{-1}(U)}\longleftrightarrow \operatorname{id}_{U_0}\times\bar\pi.
\end{equation}
Under this isomorphism, every labeled divisor with index $j=b+k$, $1\leq k\leq r-b$, corresponds to $U_0\times\bar E_k$ and the divisors $E_2,\ldots,E_b$ do not meet $\rho^{-1}(U)$.  This includes the last boundary divisor: the codimension of the residual determinantal locus is
\begin{equation}
\label{eq:cat-residual-codimension}
(m-b)+(n-b)-2(r-b)+1=m+n-2r+1=c,
\end{equation}
so the final divisor is exceptional or a retained strict transform in both constructions simultaneously.

Denote $\bar F:=\bar\pi^{-1}(0)$.  Restricting \eqref{eq:cat-stratum-local-tower} to $R_b\cap U=U_0\times\{0\}$ gives
\begin{equation}
\label{eq:cat-stratum-local-fiber}
\rho^{-1}(R_b\cap U)\simeq U_0\times\bar F,
\end{equation}
and, for $1\leq k\leq r-b$,
\begin{equation*}
\label{eq:cat-stratum-local-divisors}
E_{b+k}\cap\rho^{-1}(R_b\cap U)
\simeq U_0\times(\bar E_k\cap\bar F).
\end{equation*}
By Lemma~\ref{lem:cat-fiber}, applied to the residual problem, $\bar F=\bar E_1$.  Hence every point in $\rho^{-1}(R_b)$ belongs to $E_{b+1}$, and no such point belongs to $E_2,\ldots,E_b$. In particular, the smallest boundary index greater than $1$ at any point over $R_b$ is exactly $b+1$.  This conclusion also holds when the residual parameters are $(1,1,1)$: in that case $\bar\pi$ is the identity of $\A^1$ and $\bar F=\bar E_1=\{0\}$.

We next prove the assertion for $J=\{1\}$.  Recall that $U_r=\PP^d\setminus\PP(Y_r).$ All centers of the projective blow-up tower are contained in the
successive inverse images of $\PP(Y_r)$, and the retained final divisor, if any, is supported there as well.  Thus $\rho$ is an isomorphism
over $U_r$, and none of $E_2,\ldots,E_r$ meets $\rho^{-1}(U_r)$. Conversely, every point over $\PP(Y_r)$ lies over some $R_b$ with
$1\leq b<r$, and hence belongs to $E_{b+1}$.  Since $F=E_1$, we conclude that
\[
E_{\{1\}}^\circ\cap F
=F\setminus\bigcup_{j=2}^{r}(E_j\cap F)
=\rho^{-1}(U_r)
\xrightarrow[\ \rho\ ]{\ \sim\ }U_r.
\]
For $r=1$, the union is empty and the same conclusion follows directly
from Lemma~\ref{lem:cat-fiber}.

Now let $J=\{1=j_0<j_1<\cdots<j_t\}$ with $t\geq1$, and set $a=j_1-1,\;J'=\{j-a:j\in J\setminus\{1\}\}.$ A point of $E_J^\circ\cap F$ belongs to $E_{j_1}$ but to none of $E_2,\ldots,E_{j_1-1}$.  The preceding description of the fibers over $R_b$ shows that its image under $\rho$ has rank exactly $b=j_1-1=a$.  Thus, $\rho(E_J^\circ\cap F)\subseteq R_a.$ Notice that $j_1=a+1$, so $1\in J'$.  In the local product \eqref{eq:cat-stratum-local-fiber}, taken with $b=a$, we have
$$
E_{a+k}\cap\rho^{-1}(R_a\cap U) \simeq U_0\times(\bar E_k\cap\bar F) \quad(1\leq k\leq r-a),
$$
while $E_2,\ldots,E_a$ are absent and $E_1$ contains the whole fiber.
Under the index correspondence $E_{a+k}\leftrightarrow\bar E_k$, one gets
\begin{equation}
\label{eq:cat-stratum-local-trivialization}
(E_J^\circ\cap F)\cap\rho^{-1}(R_a\cap U) \simeq U_0\times(\bar E_{J'}^\circ\cap\bar F),
\end{equation}
where $\bar\pi$ now denotes the residual principalization for $(m-a,n-a,r-a)$.  As the isomorphism is over $U_0=R_a\cap U$, these
identifications give Zariski trivializations of \eqref{eq:cat-stratum-fibration}.  
\end{proof}
\begin{proposition}
\label{prop:cat-boundary-euler}
For every nonempty $J\subseteq\{1,\ldots,r\}$,
\begin{equation}
\label{eq:cat-boundary-euler}
\chi(E_J^\circ\cap F)
=
\begin{cases}
0,&1\notin J,\\
c\,2^{|J|-1},&1\in J.
\end{cases}
\end{equation}
\end{proposition}

\begin{proof}
The first case follows from Lemma~\ref{lem:cat-fiber}.  We prove the second by induction on $r$.  If $r=1$, then $F=\PP^{c-1}$ and its Euler
characteristic is $c$.  For arbitrary $r$, the case $J=\{1\}$ follows from Proposition~\ref{prop:cat-residual-strata} and Lemma~\ref{lem:cat-rank-euler}.

Suppose $|J|\geq2$, and use the notation of Proposition~\ref{prop:cat-residual-strata}.  The residual index $r'=r-a$ is strictly smaller than $r$, its codimension parameter is still $c$ by \eqref{eq:cat-residual-codimension}, and $|J'|=|J|-1$. Induction gives $\chi(\bar E_{J'}^\circ\cap\bar F)=c\,2^{|J|-2}.$ Multiplicativity of Euler characteristic for the Zariski locally trivial fibration \eqref{eq:cat-stratum-fibration} therefore yields
$$
\chi(E_J^\circ\cap F) =\chi(R_a)\,\chi(\bar E_{J'}^\circ\cap\bar F) =2\cdot c\,2^{|J|-2}=c\,2^{|J|-1}.
$$
This completes the induction.
\end{proof}

\subsection{Proof of the formula and the determinant hypersurface}

\begin{proof}[Proof of Theorem~\ref{thm:cat-zeta}]
By Corollary~\ref{cor:cat-principalization}, $N_js+\nu_j=D_j(s)$. The definition of the local topological zeta function gives
$$ \Ztop_{I_r,0}(s) =\sum_{\varnothing\neq J\subseteq\{1,\ldots,r\}} \chi(E_J^\circ\cap F)\prod_{j\in J}\frac1{D_j(s)}.
$$
Substituting Proposition~\ref{prop:cat-boundary-euler} and writing
$J=\{1\}\cup K$, we obtain
\begin{align*}
\Ztop_{I_r,0}(s) &=\frac{c}{D_1(s)} \sum_{K\subseteq\{2,\ldots,r\}} \prod_{j\in K}\frac{2}{D_j(s)} =\frac{c}{D_1(s)} \prod_{j=2}^{r}\left(1+\frac{2}{D_j(s)}\right).
\end{align*}
This proves \eqref{eq:cat-zeta-product}.  To obtain \eqref{eq:cat-zeta-reindexed}, put $q=r-j+1$ and observe that
$$ D_j(s)=q(s+2)+c-2, \; D_j(s)+2=q(s+2)+c. $$
As $j$ runs from $2$ to $r$, $q$ runs from $r-1$ to $1$. Hence, the formula follows.
\end{proof}

\begin{proof}[Proof of Corollary \ref{cor:cat-determinant}]
Take $m=n=r$ in Theorem~\ref{thm:cat-zeta}, so that $c=1$. The denominator roots $-2+1/q$ are distinct and strictly greater than $-2$, whereas all numerator roots are of the form $-2-1/q<-2$. Thus, there are no cancellations proving the claims on poles. Finally, the last claim follows directly from the description of the Bernstein--Sato polynomial of $f_n=\det H_{n,n}$  given in \cite[Theorem~1.1]{Lin26}.
\end{proof}

\begin{remark}
The difference from the other matrix types is the Euler characteristic of the projective exact-rank strata: here $\chi(R_a)=2$ for every $a<r$.  Consequently, every subset containing $1$ contributes, and its coefficient depends only on its cardinality.  In particular, there is no restriction to jumps of length one or two; the sum over all subsets factors directly into the product \eqref{eq:cat-zeta-product} (cf. Remark \ref{rem:decdiff}).
\end{remark}
\section{The generic determinantal case}

Let $V$ and $W$ be complex vector spaces of dimensions $n$ and $m$,
respectively, and set
$$
S_0:=\Hom(V,W)\simeq \A^{mn}.
$$
After choosing bases, $S_0$ is the affine space of generic $m\times n$
matrices. For $1\leq j\leq \min\{m,n\}+1$, set
$$
Y_j:=\{A\in S_0:\rank A\leq j-1\}.
$$
Fix an integer $1\leq r\leq \min\{m,n\},$ and let $I_r$ denote the ideal generated by the $r\times r$ minors of the generic $m\times n$ matrix. Then $V(I_r)=Y_r.$ For $1\leq j\leq r$, define
$$ D_j(s) := (r-j+1)s+(m-j+1)(n-j+1).$$
\begin{theorem}\label{thm:generic-determinantal-zeta}
With the notation above, the local topological zeta function of $I_r$
at the origin is
$$
Z^{\mathrm{top}}_{I_r,0}(s)
=
\prod_{j=1}^r
\frac{1}{
1+
\displaystyle
\frac{r-j+1}{(m-j+1)(n-j+1)}\,s
}.
$$
\end{theorem}

We use the standard resolution of generic determinantal varieties obtained by successively blowing up the strict transforms of the rank strata (cf. \cite{Vai84}, \cite{Joh03}, \cite[\S 3]{Sta03} and \cite[\S5]{Sch18}).

Put $X_0:=S_0.$ For $1\leq j\leq r-1$, let $\pi_j:X_j\longrightarrow X_{j-1}$ be the blow-up of the strict transform $Y_j^{(j-1)}$ of $Y_j$ in $X_{j-1}$. After these blow-ups, the strict transform $Y_r^{(r-1)}$ is smooth and
meets the exceptional divisor with normal crossings. Let $\pi^{(r-1)} :=\pi_1\circ\cdots\circ\pi_{r-1}:
X_{r-1}\longrightarrow S_0$ be the resulting morphism, and let $E_j\subset X_{r-1}$ denote the final strict transform of the exceptional divisor created by $\pi_j$.

\begin{proposition}
\label{prop:factorizing-generic-determinantal}
Then the total transform of $I_r$ on $X_{r-1}$ admits the
factorization
$$
I_r\OO_{X_{r-1}}=\OO_{X_{r-1}}\left(-\sum_{j=1}^{r-1}(r-j+1)E_j\right)\cdot I_{Y_r^{(r-1)}}.
$$
Here $I_{Y_r^{(r-1)}}$ denotes the ideal sheaf of the strict transform $Y_r^{(r-1)}\subset X_{r-1}$.
\end{proposition}
\begin{proof} By Remark~\ref{rem:simultaneous-local-reduction-other-types}, the local reduction of the generic determinantal ideals and the compatibility of the subsequent blow-up centers are identical to those used in the proof of Proposition~\ref{prop:factorizing-GS}, with $\Sym^2K^\vee$ replaced by \(\Hom(\ker A,\coker A)$ for an $m\times n$ matrix. The same induction therefore proves the factorization.
\end{proof}
Recall that $\codim_{S_0}(Y_r) = (m-r+1)(n-r+1).$
\begin{corollary}
\label{cor:generic-determinantal-principalization} Let $\pi^{(r-1)}:X_{r-1}\longrightarrow S_n$ be the embedded resolution above.
\begin{itemize}
\item If $(m-r+1)(n-r+1)>1,$ let $\pi_r:X_r = \Bl_{Y_r^{(r-1)}}(X_{r-1}) \longrightarrow X_{r-1}$ be the blow-up of the smooth strict transform $Y_r^{(r-1)}\subset X_{r-1},$ and let $E_r$ denote its exceptional divisor.
\item If $(m-r+1)(n-r+1)=1,$ then necessarily $m=n=r,$ so that $Y_r$ is the determinant hypersurface. In this case put $X_r:=X_{r-1},$ and $E_r:=Y_r^{(r-1)}.$
\end{itemize}
Then the composite morphism $\pi:X_r\longrightarrow S_0$ is a log-principalization of the determinantal ideal $I_r$, and
$$
I_r\OO_{X_r}=\OO_{X_r}\left(-\sum_{j=1}^{r}(r-j+1)E_j\right).
$$
\end{corollary}
\begin{proof} The proof is the same as that of Corollary \ref{cor:GS-principalization}.
\end{proof}
\begin{proposition} \label{prop: orddisfib} With the above setting,
\begin{itemize}
\item[1.]  $N_j=r-j+1$ and $\nu_j=(m-j+1)(n-j+1)$ for $1 \leq j \leq r$;
\item[2.] the fiber $F$ is the final strict transform of the exceptional divisor
created by the first blow-up of $Y_1=\{0\}.$
In particular, $F=E_1$. Consequently, $E_J^\circ\cap F=\varnothing$ unless $1\in J.$
\end{itemize}
\end{proposition}
\begin{proof}
One can repeat verbatim the proof of Lemma \ref{lem:order}, Lemma \ref{lem:nu} and Lemma \ref{lem:fiber}.
\end{proof}
Let $
J=\{1=j_0<j_1<\cdots<j_t=b\}
\subseteq\{1,\ldots,r\},$
and set $ d_\ell:=j_\ell-j_{\ell-1}$ and $ 1\leq\ell\leq t.$ Then $d_1+\cdots+d_t=b-1.$ Put $a:=n-b+1$ and $c:=m-b+1.$ Define the flag varieties
$$
\mathcal F_J^V
=
\left\{
V=K_0\supset K_1\supset\cdots\supset K_t:
\dim(K_{\ell-1}/K_\ell)=d_\ell
\right\}
$$
and
$$
\mathcal F_J^W
=
\left\{
0=I_0\subset I_1\subset\cdots\subset I_t\subset W:
\dim(I_\ell/I_{\ell-1})=d_\ell
\right\}.
$$
Set $ \mathcal F_J := \mathcal F_J^V\times\mathcal F_J^W.$ Over $\mathcal F_J$, let $\mathcal K_0\supset\cdots\supset\mathcal K_t$ and $0=\mathcal I_0\subset\cdots\subset\mathcal I_t$ denote the tautological flags. Set $\mathcal G_\ell:= \mathcal K_{\ell-1}/\mathcal K_\ell,$ and $\mathcal H_\ell := \mathcal I_\ell/\mathcal I_{\ell-1}.$
Note that both have rank $d_\ell$.

For $d\geq1$, define
$$P_d := \left\{ [A]\in\mathbf P(\Hom(\C^d,\C^d)): \rank A=d \right\}.
$$
Thus $P_d\simeq\PGL_d.$

For $a,c\geq1$ and $q\geq1$, define $$
U_{a,c,q} := \left\{ [A]\in\mathbf P(\Hom(\C^a,\C^c)):\rank A\geq q\right\}.
$$

\begin{proposition}\label{prop:generic-boundary-stratum}
The stratum $E_J^\circ\cap F$ is a Zariski locally trivial fibration over $\mathcal F_J$ with fiber $P_{d_1}\times\cdots\times P_{d_t}\times U_{a,c,r-b+1}.$
\end{proposition}

\begin{proof} The proof is the same as that of \ref{prop:strata} with the residual symmetric space \(\Sym^2K^\vee\) replaced by $\Hom(\ker A,\coker A)$ for $A$ an $m\times n$ matrix.
\end{proof}

Now, we first compute the Euler characteristic of the flag base.

\begin{lemma}\label{lem:generic-flag-euler}
One has
$$
\chi(\mathcal F_J)
=
\frac{m!n!}
{(m-b+1)!(n-b+1)!
(d_1!\cdots d_t!)^2}.
$$
\end{lemma}

\begin{proof}
It follows from the same argument as in Lemma \ref{lem:flagchi} and the fact that $\chi(F_j)=\chi(F_j^V)\cdot \chi(F_j^W)$.
\end{proof}

We next compute the Euler characteristic of the nondegenerate projective
matrix factor.

\begin{lemma}\label{lem:pgl-euler}
For $d\geq1$,
\[
\chi(P_d)
=
\begin{cases}
1,&d=1,\\
0,&d\geq2.
\end{cases}
\]
\end{lemma}

\begin{proof}
For $d=1$, the variety $P_1$ consists of a single point. Assume $d\geq2$. Let $ T=(\C^\times)^d$ act on $\Mat_{d\times d}$ by left multiplication by diagonal matrices:
$$
(t_1,\ldots,t_d)\cdot A
=
\diag(t_1,\ldots,t_d)A.
$$
This induces an action on $\mathbf P(\Mat_{d\times d}).
$

A $T$-fixed projective point must lie in a weight space. The weight
spaces consist of matrices supported in a single row. Every nonzero
matrix in such a weight space has rank at most $1$. Hence, for $d\geq2$,
the open subset $P_d$ of invertible matrices contains no $T$-fixed
points: $P_d^T=\varnothing.$
For a complex algebraic variety with an algebraic torus action, $\chi(X)=\chi(X^T).$
Therefore, $\chi(P_d)=0.$
\end{proof}

Finally, consider $$
R_{a,c,q}:=\left\{[A]\in\mathbf P(\Hom(\C^a,\C^c)):\rank A=q\right\}.$$
for $1\leq q\leq\min\{a,c\}$.
\begin{lemma}\label{lem:rectangular-rank-euler}
One has
$$
\chi(R_{a,c,q})=\binom aq\binom cq\,\chi(P_q).
$$
Consequently,
$$
\chi(U_{a,c,p})=
\begin{cases}
ac,&p=1,\\
0,&p\geq2.
\end{cases}
$$
\end{lemma}

\begin{proof}
A rank-$q$ linear map $A:\C^a\longrightarrow\C^c$ is determined by:
\begin{enumerate}
\item its kernel $K\in\Gr(a-q,a)$;
\item its image $I\in\Gr(q,c);$
\item a projective isomorphism $\C^a/K\xrightarrow{\sim}I.$
\end{enumerate}
Therefore, $ R_{a,c,q}
\longrightarrow
\Gr(a-q,a)\times\Gr(q,c)
$ is a Zariski locally trivial fibration with fiber $P_q$. Hence
$$
\chi(R_{a,c,q})
=
\chi(\Gr(a-q,a))
\chi(\Gr(q,c))
\chi(P_q)=\binom aq\binom cq\chi(P_q)=\begin{cases}
ac,&q=1,\\
0,&q\geq2.
\end{cases}.
$$
where the last equality follows from Lemma~\ref{lem:pgl-euler}.
Since $ U_{a,c,p}=\bigsqcup_{q=p}^{\min\{a,c\}}R_{a,c,q}$,
the additivity of the  Euler characteristic gives
\[
\chi(U_{a,c,p})
=
\begin{cases}
ac,&p=1,\\
0,&p\geq2.
\end{cases}
\] which ends the verification.
\end{proof}
\begin{proof}[Proof of Theorem~\ref{thm:generic-determinantal-zeta}]
Combining Proposition~\ref{prop:generic-boundary-stratum} with the
preceding Euler-characteristic computations gives
\[
\chi(E_J^\circ\cap F)
=
\chi(\mathcal F_J)
\left(
\prod_{\ell=1}^t\chi(P_{d_\ell})
\right)
\chi(U_{a,c,r-b+1}).
\]
By Lemma~\ref{lem:pgl-euler}, this vanishes unless $ d_\ell=1$
for every $\ell.$ By Lemma~\ref{lem:rectangular-rank-euler}, it also vanishes unless $ r-b+1=1,$ that is, $b=r.$ As a matter of fact, the only possible index set is $J=\{1,2,\ldots,r\}$ and thus
$$
\begin{aligned}
\chi(E_{\{1,\ldots,r\}}^\circ\cap F)
&=
\frac{n!}{(n-r+1)!}
\frac{m!}{(m-r+1)!}
(n-r+1)(m-r+1)
\\
&= \frac{m!n!}{(m-r)!(n-r)!}.
\end{aligned}
$$

By definition,
$$
Z^{\mathrm{top}}_{I_r,0}(s)
=
\sum_{\varnothing\neq J\subseteq\{1,\ldots,r\}}
\chi(E_J^\circ\cap\pi^{-1}(0))
\prod_{j\in J}\frac{1}{N_js+\nu_j}.
$$
By Proposition~\ref{prop: orddisfib}, only subsets containing $1$
can contribute. The preceding computation shows in fact that every
stratum has Euler characteristic zero except $J=\{1,\ldots,r\}.$
Therefore
\[
Z^{\mathrm{top}}_{I_r,0}(s)
=
\frac{\displaystyle
\frac{m!n!}{(m-r)!(n-r)!}}
{\displaystyle
\prod_{j=1}^rD_j(s)}=\prod_{j=1}^r
\frac{1}{
1+
\displaystyle
\frac{r-j+1}{(m-j+1)(n-j+1)}s
}
\] as desired.
\end{proof}

\section{The generic skew-symmetric determinantal case}
\label{sec:skew-symmetric-case}

Let $V$ be a complex vector space of dimension $n$, and set $ S_0:=\wedgeTwo V^\vee.$ After choosing a basis of $V$, the affine space $S_0$ is identified with the space of generic skew-symmetric $n\times n$ matrices.

Fix $1\leq r\leq \left\lfloor\frac n2\right\rfloor.$ For $1\leq j\leq r$, let $I_j$ denote the ideal generated by the Pfaffians of size $2j$ of the generic skew-symmetric matrix. Its zero locus is
$$ Y_j := \left\{\omega\in\wedgeTwo V^\vee:\rank(\omega)\leq 2j-2\right\}.$$
Thus $$ Y_1=\{0\} \subset Y_2 \subset\cdots\subset Y_r.$$
Define also
$$ D_j(s):=(r-j+1)s +\binom{n-2j+2}{2}.
$$
\begin{theorem}
\label{thm:skew-topological-zeta}
The local topological zeta function of the Pfaffian ideal $I_r$ at the
origin is 
$$\Ztop_{I_r,0}(s)=\prod_{j=1}^r\frac{1}{1+\displaystyle
\frac{r-j+1}{\binom{n-2j+2}{2}}\,s}.
$$ 
\end{theorem}
\begin{remark}
Our formula differs from Chen-Shi-Xu-Zuo's one by a factor $2^{-\left\lfloor\frac n2\right\rfloor}$ (cf. \cite[\S3.8]{CSXZ25})
\end{remark}
We use the Gaube--Schober resolution of the reduced skew-symmetric
determinantal loci (see \cite[Theorem A]{GS24}). Put $X_0:=S_0.$ For $1\leq j\leq r-1$, let $
\pi_j:X_j\longrightarrow X_{j-1}$ be the blow-up of the strict transform $Y_j^{(j-1)}$ of $Y_j$ in
$X_{j-1}$. After these blow-ups the strict transform
$$
Y_r^{(r-1)}\subset X_{r-1}
$$ of the reduced variety $Y_r$ is smooth and meets the exceptional divisor with simple normal crossings. Note that $$
\codim_{S_0}(Y_r)=\binom{n-2r+2}{2}.
$$ Let $\pi^{(r-1)} := \pi_1\circ\cdots\circ\pi_{r-1} :X_{r-1}\longrightarrow S_0$ be the resulting morphism, and let \(E_j\subset X_{r-1}\) denote the final strict transform of the exceptional divisor created by $\pi_j$.
\begin{proposition}
\label{prop:factorizing-skew-symmetric}
The total transform of the Pfaffian ideal \(I_r\) on
\(X_{r-1}\) admits the factorization
$$
I_r\OO_{X_{r-1}}=\OO_{X_{r-1}}\left(-\sum_{j=1}^{r-1}(r-j+1)E_j\right)\cdot I_{Y_r^{(r-1)}}.
$$
Here \(I_{Y_r^{(r-1)}}\) denotes the ideal sheaf of the strict transform $ Y_r^{(r-1)}\subset X_{r-1}.$
\end{proposition}
\begin{proof}
The proof is the same as that of Proposition \ref{prop:factorizing-GS}.
\end{proof}

\begin{corollary}
\label{cor:skew-symmetric-principalization} Let $\pi^{(r-1)}:X_{r-1}\longrightarrow S_n$ be the Gaube--Schober embedded resolution of reduced skew-symmetric determinantal loci above. 
\begin{itemize}
\item If $\codim_{S_0}(Y_r)>1,$ let $ \pi_r:X_r=\Bl_{Y_r^{(r-1)}}(X_{r-1})\longrightarrow X_{r-1} $
be the blow-up of the smooth strict transform $Y_r^{(r-1)}\subset X_{r-1},$ and let $E_r$ denote its exceptional divisor.

\item If $\codim_{S_0}(Y_r)=1,$ then necessarily $n=2r,$ and $Y_r$ is the hypersurface defined by the maximal Pfaffian. In
this case put $X_r:=X_{r-1},\;E_r:=Y_r^{(r-1)}.$

\end{itemize}
Then the composite morphism $\pi:X_r\longrightarrow S_0$ is a log-principalization of the Pfaffian ideal $I_r$, and
$$
I_r\OO_{X_r}=
\OO_{X_r}
\left(
-\sum_{j=1}^{r}(r-j+1)E_j
\right).
$$
\end{corollary}
\begin{proof}
The proof is the same as that of Corollary \ref{cor:GS-principalization}.
\end{proof}
As in the symmetric and the generic case, we have the following similar results.
\begin{lemma}
\label{lem:skew-multiplicity}
\begin{itemize}
\item[1.] For every $1\leq j\leq r$, $N_j=r-j+1$ and $\nu_j=\binom{n-2j+2}{2}$.
\item[2.] The fiber $F$ is the final strict transform of the exceptional divisor created by the first blow-up of $Y_1=\{0\}$. In particular, $F=E_1.$ Consequently, 
$E_J^\circ\cap F=\varnothing$ unless $1\in J.$
\end{itemize}
\end{lemma}
Fix $J=\{1=j_0<j_1<\cdots<j_t=b\}\subseteq\{1,\ldots,r\},$ and put $d_\ell := j_\ell-j_{\ell-1},$ $1\leq\ell\leq t.$
Then $d_1+\cdots+d_t=b-1.$ Set $m:=n-2b+2.$ Define the flag variety
$$
\cF_J := \left\{ V=K_0\supset K_1\supset\cdots\supset K_t: \dim(K_{\ell-1}/K_\ell)=2d_\ell\right\}.
$$

Let $$ V\otimes\cO_{\cF_J} =\cK_0\supset\cK_1\supset\cdots\supset\cK_t$$
be the tautological flag, and set $\cG_\ell:=\cK_{\ell-1}/\cK_\ell.$
Then $\rank(\cG_\ell)=2d_\ell,$ and  $\rank(\cK_t)=m.$

For $d\geq1$, define
$$A_{2d}:=\left\{[\omega]\in\PP(\wedgeTwo\C^{2d\,\vee}):\rank(\omega)=2d
\right\}.
$$
Thus $A_{2d}$ is the variety of nondegenerate projective alternating
forms on $\C^{2d}$.

For $m\geq2$ and $a\geq1$, define
$$
U^{\mathrm{skew}}_{m,a}:=\left\{[\omega]\in\PP(\wedgeTwo\C^{m\,\vee}):\rank(\omega)\geq2a
\right\}.
$$

\begin{proposition}
\label{prop:skew-boundary-strata}
The stratum $E_J^\circ\cap F$ is a Zariski locally trivial fibration over $\cF_J$ with fiber
$$
A_{2d_1}\times\cdots\times A_{2d_t}
\times
U^{\mathrm{skew}}_{m_J,r-b+1}.
$$
\end{proposition}
\begin{proof}
The proof is the same as that of Proposition \ref{prop:strata} with the residual symmetric space $\Sym^2K^\vee$ replaced by the residual skew-symmetric space $\bigwedge^2K^\vee$.
\end{proof}

Similar computations show the following Euler characteristics.
\begin{lemma}
\label{lem:skew-flag-euler}
With the notations above,
\begin{itemize}
\item[1.]$$\chi(\cF_J)=\frac{n!}{m!\,(2d_1)!\cdots(2d_t)!}.$$
\item[2.] $$
\chi(A_{2d})=\begin{cases}
1,&d=1,\\
0,&d\geq2.
\end{cases}
$$

\item[3.] $$
\chi\left(U^{\mathrm{skew}}_{m,a}\right)
=
\begin{cases}
\binom m2,&a=1,\\
0,&a\geq2.
\end{cases}
$$
\end{itemize}
\end{lemma}

\begin{proof}[Proof of Theorem~\ref{thm:skew-topological-zeta}] By Proposition~\ref{prop:skew-boundary-strata},
$$
\chi(E_J^\circ\cap F)
=
\chi(\cF_J)
\left(
\prod_{\ell=1}^t
\chi(A_{2d_\ell})
\right)
\chi\left(
U^{\mathrm{skew}}_{m_J,r-b+1}
\right).
$$
Lemma~\ref{lem:skew-flag-euler}.2 implies that this expression vanishes unless $ d_\ell=1$ for every $\ell$. By Lemma~\ref{lem:skew-flag-euler}.3, it also vanishes unless $ r-b+1=1,$
that is, $b=r.$ Therefore the only subset $J$ giving a nonzero contribution is $J=\{1,2,\ldots,r\}.$ For this subset, $d_1=\cdots=d_{r-1}=1$ and $m=n-2r+2.$ Hence
$$
\begin{aligned}
\chi(E_{\{1,\ldots,r\}}^\circ\cap F)
&=
\frac{n!}{
2^{r-1}(n-2r+2)!
}
\binom{n-2r+2}{2}
\\
&=
\frac{n!}{
2^{r-1}(n-2r+2)!
}
\cdot
\frac{(n-2r+2)(n-2r+1)}{2}
\\
&= \frac{n!}{2^r(n-2r)!}.
\end{aligned}
$$
By Lemma~\ref{lem:skew-multiplicity}.2, only subsets containing $1$ can contribute. The Euler-characteristic
calculation above shows that every contribution vanishes
except the one corresponding to $J=\{1,\ldots,r\}.$ Therefore
$$
\Ztop_{I_r,0}(s)
=
\frac{
\displaystyle
\frac{n!}{2^r(n-2r)!}
}{
\displaystyle
\prod_{j=1}^r
\left(
(r-j+1)s+\binom{n-2j+2}{2}
\right)
}= \prod_{j=1}^r
\frac{1}{
1+
\displaystyle
\frac{r-j+1}{\binom{n-2j+2}{2}}\,s
}
$$
as desired.
\end{proof}

\begin{remark} \label{rem:decdiff}
The decisive difference between the generic, the symmetric, and the skew-symmetric determinantal cases is the Euler characteristic of the projective nondegenerate factor.

For generic matrices, this factor is $P_d\simeq\PGL_d,$ and $\chi(P_d)=0$ for all $d\geq2.$
Thus only jumps of length $1$ survive, so there is a unique contributing
path $1\longrightarrow2\longrightarrow\cdots\longrightarrow r.$

For skew-symmetric matrices, the nondegenerate factor $A_{2d}$ satisfies
$\chi(A_{2d})=0$ for all $d\geq2.$ Hence only jumps of length $1$ in the Pfaffian indexing survive.
So, once again, only this path $1\longrightarrow2\longrightarrow\cdots\longrightarrow r$
contributes to the Denef--Loeser formula.

For symmetric matrices, the corresponding factor is the space $Q_d$ of
nondegenerate projective quadratic forms, for which $\chi(Q_1)=\chi(Q_2)=1,$ and $\chi(Q_d)=0\quad(d\geq3).$
Therefore both jumps of length $1$ and $2$ survive, producing the recurrence.

$$
\begin{array}{c|c|c}
\text{Matrix type} & \text{Nondegenerate projective factor} & \text{Surviving jumps} \\[2mm]\hline
\text{generic} & \PGL_d & 1 \\[1mm] \text{skew-symmetric} & A_{2d} & 1 \\
[1mm] \text{symmetric} & Q_d & 1,2
\end{array}
$$
\end{remark}

\end{document}